\documentclass[12pt]{amsart}
\usepackage{amsmath}
\usepackage{amsxtra}
\usepackage{amscd}
\usepackage{amsthm}
\usepackage{amsfonts}
\usepackage{amssymb}
\usepackage{eucal}
\usepackage{epsfig}
\usepackage{graphics}
\usepackage{graphicx}
\usepackage{accents}
\usepackage{dynkin-diagrams}
\usepackage{comment}
\usepackage{enumitem}

\usepackage{mathtools}
\usepackage{tikz-cd}
\usepackage{mathtools}
\usepackage{pgf}
\usetikzlibrary{arrows, matrix}

\usepackage{here} 
\usepackage{simplewick}
\numberwithin{equation}{section}
\newtheorem{thm}{Theorem}[section]
\newtheorem{prop}[thm]{Proposition}
\newtheorem{lem}[thm]{Lemma}
\newtheorem{rem}[thm]{Remark}

\theoremstyle{definition}

\newcommand{\bra}[1]{\langle #1 |}        
\newcommand{\ket}[1]{{| #1 \rangle}}

\newcommand{\C}{{\mathbb C}}
\newcommand{\Z}{{\mathbb Z}}

\newcommand{\al}{\alpha}

\newcommand{\GS}{\mathfrak{S}}

\newcommand{\bs}{\boldsymbol}

\newcommand{\gl}{\mathfrak{gl}}

\newcommand{\sln}{\mathfrak{sl}}   

\newcommand{\la}{\lambda}

\newcommand{\res}{{\rm res}}
\newcommand{\Res}{\mathop{\rm res}}

\newcommand{\Sym}{\mathrm{Sym}}

\newcommand{\E}{\mathcal{E}}
\newcommand{\on}{\operatorname}

\newcommand{\beq}{\begin{equation}}
\newcommand{\eeq}{\end{equation}}
\newcommand{\be}{\begin{equation*}}
\newcommand{\ee}{\end{equation*}}

\def\bbox at (#1,#2){\filldraw[fill=gray,-] (#1,#2)--++(0.6,0)--++(-0.4,-0.3)--++(-0.6,0)--(#1,#2);}
\def\rbbox at (#1,#2){\filldraw[fill=gray,-] (#1,#2)--++(0,0.6)--++(0.4,0.3)--++(0,-0.6)--(#1,#2);}
\def\fbbox at (#1,#2){\filldraw[fill=gray,-] (#1,#2)--++(0,0.6)--++(0.6,0)--++(0,-0.6)--(#1,#2);}
\def\wbox at (#1,#2){\draw[fill=white,-] (#1,#2)--++(0.6,0)--++(-0.4,-0.3)--++(-0.6,0)--(#1,#2);}
\def\rwbox at (#1,#2){\draw[fill=white,-] (#1,#2)--++(0,0.6)--++(0.4,0.3)--++(0,-0.6)--(#1,#2);}
\def\fwbox at (#1,#2){\draw[fill=white,-] (#1,#2)--++(0,0.6)--++(0.6,0)--++(0,-0.6)--(#1,#2);}
\def\dbox at (#1,#2){\draw[dashed] (#1,#2)--++(0.6,0)--++(-0.4,-0.3)--++(-0.6,0)--(#1,#2);}
\def\rdbox at (#1,#2){\draw[dashed] (#1,#2)--++(0,0.6)--++(0.4,0.3)--++(0,-0.6)--(#1,#2);}
\def\fdbox at (#1,#2){\draw[dashed] (#1,#2)--++(0,0.6)--++(0.6,0)--++(0,-0.6)--(#1,#2);}

\begin{document}
\begin{title}[Combinatorial modules]
{Combinatorial bases in quantum toroidal $\gl_2$ modules}
\end{title}

\author{M. Jimbo and E. Mukhin}

\address{MJ: 
Professor Emeritus,
Rikkyo University, Toshima-ku, Tokyo 171-8501, Japan}
\email{jimbomm@rikkyo.ac.jp}
\address{EM: Department of Mathematics,
Indiana University Indianapolis,
402 N. Blackford St., LD 270, 
Indianapolis, IN 46202, USA}
\email{emukhin@iu.edu}

\begin{abstract}
    We show that many tame modules of the quantum toroidal $\gl_2$ 
algebra
can be explicitly constructed in a purely combinatorial way using the theory of $q$-characters. The examples include families of evaluation modules obtained from analytic continuation and automorphism twists of Verma modules of the quantum affine $\gl_2$ algebra. 
The combinatorial bases in the modules are labeled by colored plane partitions with various properties.
\end{abstract}
\maketitle

\section{Introduction}
The problem of finding a good basis in a representation is a central problem 
in representation theory. 
There are various approaches depending on the definition of the word  ``good". The theory of canonical \cite{L90} and crystal bases \cite{K90} in representations of quantum algebras and quantum affine algebras is compatible with the change $q\to q^{-1}$ 
in an 
appropriate sense. The Gelfand-Tsetlin bases in symmetric groups, see e.g. \cite{OV05}, and in
Lie algebras, see \cite{GT50}, 
respect the branching rules. The bases of Bethe vectors in various tensor products are eigenbases for a family of commuting Hamiltonians, \cite{B31}.

In this note, we discuss bases which we call combinatorial. 
A quantum affine algebra at level zero has a family of
commutative Cartan currents. If a module is tame, that is if the Cartan currents are diagonalizable and each eigenspace is one-dimensional, we get (up to normalization) a combinatorial basis of eigenvectors.


In combinatorial bases, the action of the quantum affine group can be written explicitly up to some additional constants(depending on the normalization of the basis) since the matrix coefficients of generating series are constant multiples of delta functions. In \cite{FJM2} the set of these additional constants is called the supplement.

 If a module is tame, the combinatorial basis (that is the set of eigenvalues of Cartan currents or the $q$-character) can be constructed recursively by the algorithm described in \cite{FM1}. To justify that the algorithm produces a correct result it is sufficient to fix the additional constants or the supplement. Solving the recursion and finding the supplement is one of the subjects of this text. We note that if a module is tame, the algorithm of \cite{FM1} does seem to produce the correct answer in all cases.
 
In type $A$, the combinatorial basis in all known cases is labeled by colored plane partitions with various boundary conditions. The creation and annihilation operators of a given color add or delete a box of the same color with  
explicit coefficients.
We develop a general scheme of fixing the supplement, see \eqref{constants 1}, \eqref{constants 2}, and Theorem \ref{main thm}. Our method is based on viewing plane partitions as 
``tensor products" of boxes in a lexicographic order. Such intuition comes from the construction of representations developed in \cite{FJMM1}.

In general, the problem of determining the supplement is not solved. See, for example, Sections 5.2 and  5.3 of \cite{FJM2} for the discussion of this issue.

\medskip 

Sometimes, one can construct a combinatorial basis in a 
module over a quantum algebra even if the level of the module is not zero or if the eigenvalues of Cartan currents are not distinct. It happens if one manages to upgrade the action of the quantum algebra to a tame level zero module over an affinization of the quantum algebra. 
We list known examples.

\begin{itemize}
\item Any finite--dimensional irreducible $U_q\gl_n$ module can be considered as an evaluation module over $U_q\widehat\gl_n$. The evaluation modules are tame and therefore one gets a combinatorial basis depending on a
 parameter. The $q$-character is known, Lemma 4.7 in \cite{FM2}, which gives 
a
parameterization of the combinatorial basis  by semi-standard Young tableaux.  A semi-standard Young tableau can be interpreted as a plane partition over 
a Young diagram
of height at most $n$ with strictly increasing height along columns.

\item The modules of minimal models of deformed affine $W$-algebras of type $\gl_n$. In this case, the combinatorial answer is the cylindrical partitions, see  \cite{FFJMM}.  It includes the deformed Virasoro algebra 
(multiplied by an
additional Heisenberg current) 
when $n=2$.

\item  Integrable irreducible  $U_q\widehat{\gl}_n$-modules of highest weight $(0,\dots,0,q,0,\dots,0)$ can be extended to a module over the quantum toroidal algebra of type $\gl_n$. This extension depends on a parameter $q_3$ of the toroidal algebra and it is an evaluation module only for $q_3^n=q^2$. The combinatorial basis is parameterized by 
colored partitions. The corresponding modules over toroidal algebras are known as Fock modules, 
\cite{FJMM1}.

\item  The supersymmetric analog of the previous item is also known, see \cite{BM}. The combinatorial basis is parameterized by 
colored super partitions.
\end{itemize}

We extend this list for the case of $\gl_2$ as follows.

\begin{itemize}
\item  The $U_q\widehat{\gl}_2$ Verma modules. The corresponding evaluation module over the quantum
toroidal algebra is tame. We compute the corresponding basis and 
matrix coefficients. 
The combinatorial set is the set of pairs of vertical partitions, see Figure \ref{verma picture}.

\item  The relaxed $U_q\widehat{\gl}_2$ Verma modules. These are modules of generic level induced from a ``one line" representation of $U_q{\gl}_2$. The corresponding weight diagram is given in Figure \ref{flat pic}. The evaluation 
module over the quantum toroidal algebra is tame and we compute the corresponding basis and  matrix coefficients. 
The combinatorial set is the set of 
pairs of vertical partitions with an additional single tower over a box,
see Figure \ref{relaxed verma picture}. 
 
\item Slanted slope $m$ relaxed $U_q\widehat{\gl}_2$ Verma modules, $m\in \Z$. These modules are obtained from the relaxed $U_q\widehat{\gl}_2$ Verma modules via a twist by an automorphism.
The corresponding weight diagram is given in Figure \ref{slanted pic}. The evaluation module over the quantum
toroidal algebra is tame and we compute the corresponding basis and 
matrix coefficients. 
The combinatorial set is the set of pairs of vertical partitions with an additional tower over a staircase consisting of $2|m|+1$ boxes, see Figure \ref{slanted relaxed verma picture}. 

\end{itemize}

One may wonder how rich the class of modules with combinatorial bases is. 
In other words, how many tame modules do we have? Tame modules for $U_q\widehat\gl_n$ have been classified in \cite{NT} and they correspond to (unions of) skew Young diagrams. Tame modules for $U_q\hat{\mathfrak o}_{2n+1}$ (type B) are also known, \cite{BM}, the corresponding combinatorial set is given by certain non-intersecting paths, \cite{MY}. From the examples above, it is clear that for quantum toroidal algebras the set of tame modules is quite
large and interesting and deserves further study.

Many constructions discussed in the paper can be applied to the case of $\gl_n$, $n>2$. We hope to address  combinatorial bases of $\gl_n$ type in a separate publication.

We also expect, that one can define an affine crystal  structure on the combinatorial sets corresponding to slanted relaxed Verma modules.

\medskip

Our initial motivation was 
the study of integrable systems 
whose Hamiltonians are transfer matrices corresponding to quantum affine and quantum toroidal algebras. For 
highest weight modules, there is a powerful Bethe ansatz method to study the spectrum of the model. The relaxed $U_q\widehat{\gl}_2$ Verma modules
form the most natural class of non-highest weight modules of generic level for the quantum affine algebra, see \cite{FST}. 
One can consider them as analytic continuation of Verma modules and therefore one could try to understand the integrable system using continuation of the Bethe ansatz. 

Combinatorially, the slanted relaxed Verma modules look as a natural generalization, however, it is not clear how to attack the corresponding integrable system (except for the cases of $m=0,1$).

\medskip

The paper is constructed as follows. In Section \ref{sec:finite} we recall the quantum algebra $U_q\gl_2$ and the quantum affine algebra $U_q\widehat{\gl}_2$ and 
give a few simple properties. In particular, in Section \ref{sec:rep} we introduce the slanted relaxed Verma modules obtained from Verma modules by analytic continuation and automorphism twists. In Section \ref{sec:tor}, we discuss the quantum toroidal algebra $\E$ associated to $\gl_2$. In particular, in Section \ref{sec:combinatorial} we describe a way to construct combinatorial $\E$-modules, resulting in Theorem \ref{main thm} which we prove in Section \ref{proof subsection}. In Section \ref{sec:examples} we discuss various combinatorial $\E$-modules.

\medskip




\section{The quantum affine algebra associated to $\gl_2$}\label{sec:finite}

Fix $q,d\in\C^{\times}$
and set $q_1=q^{-1}d, q_2=q^2, q_3=q^{-1}d^{-1}$. Then $q_1q_2q_3=1$. 

Fix $\log q, \log d\in\C$, so that $q=e^{\log q}, d=e^{\log d}$. 

We assume that  $q,d$ are generic: for rational numbers $a,b$, the equality $q^ad^b=1$ holds if and only if $a=b=0$.

We use the standard notation $[A,B]_p=AB-pBA$ and $[r]=(q^r-q^{-r})/(q-q^{-1})$.

\subsection{Quantum algebra $U_q{\gl}_2$}
Let $\C P=\oplus_{i\in \Z/2\Z} \C\varepsilon_i$ be a $2$-dimensional vector space with the chosen basis and a non-degenerate bilinear form such that $(\varepsilon_i,\varepsilon_j)=\delta_{i,j}$. Call the lattice $P=\oplus_{i\in \Z/2\Z}\Z\varepsilon_i$ the weight lattice. Let 
$\al=\varepsilon_1-\varepsilon_2$.

The quantum $\gl_2$ algebra $U_q{\gl}_2$ 
 has generators $e_1$, $f_1$, $q^h$, $h\in P$,
  with the defining relations
  \begin{align*}
  &q^{h}q^{h'}=q^{h+h'}, \qquad q^0=1, \qquad q^{h}e_1=q^{(h,\al)}e_1q^{h},
\qquad q^{h}f_1=q^{-(h, \al)}f_1q^{h}\,, 
  &[e_1,f_1]=\frac{K-K^{-1}}{q-q^{-1}}\,,
  \end{align*} 
where $K=q^{\al}$.

The quantum $\sln_2$ algebra $U_q{\sln}_2$ is the subalgebra of  $U_q{\gl}_2$ generated by $e_1,f_1$, $K^{\pm 1}$.
  
The element ${\sf t}=q^{\varepsilon_1+\varepsilon_2}\in U_q{\gl}_2$ is central and split.

The Casimir element 
$\mathcal C=qK+q^{-1}K^{-1}+(q-q^{-1})^2f_1e_1$
is central in $U_q\gl_2$.

\subsection{Quantum affine algebra $U_q\widehat{\gl}_2$}\label{sec:affine}

The quantum affine algebra $U_q\widehat{\gl}_2$ in the 
Drinfeld new  realization is defined
 by generators $x^\pm_{k}$, $h_{j,r}$, $q^h$, $C^{\pm1}$, where
$j=1,2$, $k\in\Z$, $r\in\Z\setminus\{0\}$,   $h\in P$,  with the defining relations
\begin{align*}
&\text{$C^{\pm 1}$ are central},\quad CC^{-1}=1,
\quad q^{h}q^{h'}=q^{h+h'},\quad
q^0=1\,,\quad [q^h,h_{j,r}]=0\,, \\
&q^hx^\pm(z)q^{-h}=q^{\pm(h,\alpha)}x^\pm(z)\, ,
\\
&[h_{i,r},h_{j,s}]=\delta_{r+s,0}
\,\bar a_{i,j}(r)\,
\frac{C^r-C^{-r}}{q-q^{-1}}\,,
\\
&[h_{i,r},x^{\pm}(z)]=\pm
\,\bar a_{i,1}(r)\,
C^{-(r\pm|r|)/2}z^rx^\pm(z)\,,
\\
&[x^+(z),x^-(w)]=\frac{1}{q-q^{-1}}\Bigl(
\delta\bigl(C\frac{w}{z}\bigr)\phi^+(w)-
\delta\bigl(C\frac{z}{w}\bigr)\phi^-(z)
\Bigr)\,,
\\
&(z-q^{\pm 2}w)x^\pm(z)x^\pm(w)+
(w-q^{\pm 2}z)x^\pm(w)x^\pm(z)=0\,.
\end{align*}
Here we set $x^\pm(z)=\sum_{k\in\Z}x^\pm_{k}z^{-k}$, 
$\phi^{\pm}(z)=K^{\pm1}
\exp\bigl(\pm(q-q^{-1})\sum_{r>0}h_{1,\pm r}z^{\mp r}\bigr)$, where
$K=q^{\al}$
and 
$$
\bar a_{i,j}(r)=\frac{[r]}{r}\bigl((q^r+q^{-r})\delta_{i,j}-2(1-\delta_{i,j})\bigr)\,.
$$

\medskip
The element ${\sf t}=q^{\varepsilon_1+\varepsilon_2}$ is central and split.

The subalgebra of $U_{q}\widehat{\gl}_2$ generated by $x^{\pm}_{0}$, $q^h$,  $h\in P$, is isomorphic to $U_q\gl_2$.

The subalgebra of $U_q\widehat{\gl}_2$ generated by $x^{\pm}_{k}, h_{1,r}, K^{\pm1}$, $k\in \Z$, 
$r\in\Z\setminus\{0\}$,  and $C^{\pm 1}$, is the quantum affine $\sln_2$  algebra $U_q\widehat{\sln}_2$.
  
Algebra  $U_{q}\widehat{\gl}_2$ contains a Heisenberg subalgebra commuting with  $U_q\widehat{\sln}_2$, generated by the elements $Z_r$, $r\in\Z\setminus\{0\}$, given by
\begin{align*}
Z_r=\frac{2}{q^r-q^{-r}}
h_{1,r}+
\frac{q^r+q^{-r}}{q^r-q^{-r}}
h_{2,r}\,,
\quad [Z_r,Z_s]=-\delta_{r+s,0}[2r]\frac{1}{r}\frac{C^r-C^{-r}}{q-q^{-1}}\,.
\end{align*}

\medskip

For $\kappa\in\C^\times$ we denote $U_{q,\kappa}\widehat{\gl}_2$ 
the quotient of {$U_q\widehat{\gl}_2$ } 
by the relation $C=\kappa$.

The algebra ${U}_{q,\kappa}\widehat{\gl}_2$ is $\Z^2$ graded:
\begin{align}\label{degree}
\deg x^\pm_k=(\pm1,k), \qquad \deg h_{i,k}=(0,k), 
\qquad \deg q^h=(0,0).
\end{align}
We call the first component of the degree the weight\footnote{Often it is called ``spin" instead.} and the second component the homogeneous degree.

We denote $\widetilde{U}_{q,\kappa}\widehat{\gl}_2$ the completion of the algebra $U_{q,\kappa}\widehat{\gl}_2$ with respect to the homogeneous degree in the negative direction. 


\medskip 

We have several automorphisms of $U_q\widehat{\gl}_2$.  

For $a\in\C^\times$, the shift of spectral parameter automorphism $\tau_a:\ U_q\widehat{\gl}_2\to U_q\widehat{\gl}_2$ maps 
\beq\label{shift}
x^\pm(z)\mapsto x^\pm(z/a), \qquad \phi^\pm(z)\mapsto \phi^\pm(z/a), \qquad q^h\mapsto q^h,\qquad C\to C.
\eeq
The automorphism $\tau_a$ preserves the degree.

The automorphism $\omega:\ U_q\widehat{\gl}_2\to U_q\widehat{\gl}_2$ maps
\be
x^\pm(z)\mapsto z^{\pm 1} x^\pm(z), \qquad \phi^\pm(z)\mapsto \phi^\pm(z), \qquad q^h\mapsto q^h,\qquad C\to C.
\ee 
The automorphism $\omega$ preserves the weight but not the homogeneous degree.

\subsection{Examples of representations.}\label{sec:rep}
We start with 
representations of $U_q\gl_2$.

Let $L$ be a $\C$-vector space with basis $\{v_i\}_{i\in\Z}$. 
For $a,b,c\in\C$,
let 
\be
f_1v_i=v_{i+1},\qquad K v_0=q^a v_0, \qquad {\sf t}\,v_0=q^bv_0,\qquad  e_1v_0=cv_{-1}, 
\ee 
where $i\in\Z$. Then for generic $a,b,c$, this data is uniquely extended to a  $U_q\gl_2$-module structure on $L$. Then one can analytically continue this structure to
all values of $a,b,c$.
We call the resulting module $L(a,b,c)$. 

Given $a,b$, for generic $c$, the module $L(a,b,c)$ is irreducible. If $c=0$, then $L$ has a submodule isomorphic to 
$U_q\gl_2$ Verma module with highest weight $q^a$.  Following \cite{FST}, we call $L(a,b,c)$ a relaxed Verma module.

Note that in $L(a,b,c)$, the Casimir  element acts by the constant $\mathcal C=q^{a+1}+q^{-a-1}+(q-q^{-1})^2c$. 
Thus, $L(a,b,c)$ is uniquely determined by the values of central elements ${\sf t}, \mathcal C$ and by the weight of $v_0$. 

The two modules $L(a,b,c)$ and $L(a',b',c')$ are isomorphic if and only if values of central elements ${\sf t}, \mathcal C$ are equal, and $a-a'\in 2\Z$.

\medskip

Let $U^+\subset  U_q\widehat{\gl}_2$ be the subalgebra generated by $q^h, x^\pm_k, h_{j,k+1}$ with $k\geq 0$, $j=0,1$, $h\in P$, and $C$. Let $\kappa\in \C^\times$. 
We extend the action of $U_q\gl_2$ on $L(a,b,c)$ to $U^+$ 
action by setting $C=\kappa$ and
\be
x^\pm_k=h_{j,k}=0, \qquad k>0.
\ee 
Set  
\be
\hat L(a,b,c;\kappa)=\on{Ind}_{U^+}^{U_q\widehat{\gl}_2} L(a,b,c).
\ee 
Given $a,b$, for generic $c,\kappa$, 
the $U_q\widehat{\gl}_2$-module $\hat L(a,b,c;\kappa)$ is irreducible. If $c=0$, then it has a submodule isomorphic to 
$U_q\widehat{\gl}_2$ Verma module with highest weight $(\kappa q^{-a}, q^a)$. 
Following \cite{FST}, we call $\hat L(a,b,c;\kappa)$ a relaxed $U_q\widehat{\gl}_2$ Verma module.

\begin{figure}[ht]
\centering
\begin{tikzpicture}
\draw[dashed, ->] (-6,0) -- (6,0);
\node at (5.5,-0.4) {weight};
\draw[dashed, <-] (-5.8,0.4) -- (-5.8,-3.8);
\node at (-7.0,-3.3) {hom deg};
\coordinate (a1) at (-4.5,0);
\coordinate (a2) at (-3,0);
\coordinate (a3) at (-1.5,0);
\coordinate (a4) at (0,0);
\coordinate (a5) at (1.5,0);
\coordinate (a6) at (3,0);

\node[above] at (a1) {\tiny $-3$}; 
\node[above] at (a2) {\tiny $-2$};
\node[above] at (a3) {\tiny $-1$};
\node[above] at (a4) {\tiny $0$};
\node[above] at (a5) {\tiny $1$};
\node[above] at (a6) {\tiny $2$};

\node at (-6.2,-0.7) {\tiny $-1$};
\draw (-5.9,-0.7) -- (-5.7,-0.7);
\node at (-6.2,-1.4) {\tiny $-2$};
\draw (-5.9,-1.4) -- (-5.7,-1.4);
\node at (-6.2,-2.1) {\tiny $-3$};
\draw (-5.9,-2.1) -- (-5.7,-2.1);
\node at (-6.2,-2.8) {\tiny $-4$};
\draw (-5.9,-2.8) -- (-5.7,-2.8);

\draw[fill,black] (a2) circle [radius=0.08];
\draw[fill,black] (a3) circle [radius=0.08];
\draw[fill,black] (a4) circle [radius=0.08];
\draw[fill,black] (a5) circle [radius=0.08];
\draw[fill,black] (a6) circle [radius=0.08];

\draw[dashed] (a2) --++ (0.5,-2.5);
\draw[dashed] (a2) --++ (-0.5,-2.5);
\draw[dashed] (a3) --++ (0.5,-2.5);
\draw[dashed] (a3) --++ (-0.5,-2.5);
\draw[dashed] (a4) --++ (0.5,-2.5);
\draw[dashed] (a4) --++ (-0.5,-2.5);
\draw[dashed] (a5) --++ (0.5,-2.5);
\draw[dashed] (a5) --++ (-0.5,-2.5);
\draw[dashed] (a6) --++ (0.5,-2.5);
\draw[dashed] (a6) --++ (-0.5,-2.5);


\node at ($(a2)-(0,0.7)$) {\tiny${\bf 4}$};
\node at ($(a2)-(0,1.4)$) {\tiny${\bf 14}$};
\node at ($(a2)-(0,2.1)$) {\tiny${\bf 40}$};
\node at ($(a2)-(0,2.7)$) {$\vdots$};

\node at ($(a3)-(0,0.7)$) {\tiny${\bf 4}$};
\node at ($(a3)-(0,1.4)$) {\tiny${\bf 14}$};
\node at ($(a3)-(0,2.1)$) {\tiny${\bf 40}$};
\node at ($(a3)-(0,2.7)$) {$\vdots$};

\node at ($(a4)-(0,0.7)$) {\tiny${\bf 4}$};
\node at ($(a4)-(0,1.4)$) {\tiny${\bf 14}$};
\node at ($(a4)-(0,2.1)$) {\tiny${\bf 40}$};
\node at ($(a4)-(0,2.7)$) {$\vdots$};

\node at ($(a5)-(0,0.7)$) {\tiny${\bf 4}$};
\node at ($(a5)-(0,1.4)$) {\tiny${\bf 14}$};
\node at ($(a5)-(0,2.1)$) {\tiny${\bf 40}$};
\node at ($(a5)-(0,2.7)$) {$\vdots$};

\node at ($(a6)-(0,0.7)$) {\tiny${\bf 4}$};
\node at ($(a6)-(0,1.4)$) {\tiny${\bf 14}$};
\node at ($(a6)-(0,2.1)$) {\tiny${\bf 40}$};
\node at ($(a6)-(0,2.7)$) {$\vdots$};

\node at (4.5,-1) {$\dots$};
\node at (4.5,-1.5) {$\dots$}; 
\node at (4.5,-2) {$\dots$}; 

\node at (-4.5,-1) {$\dots$};
\node at (-4.5,-1.5) {$\dots$}; 
\node at (-4.5,-2) {$\dots$};

\end{tikzpicture}
\caption{Grading in a relaxed $U_q\widehat{\gl}_2$ Verma module.}
\label{flat pic}
\end{figure}
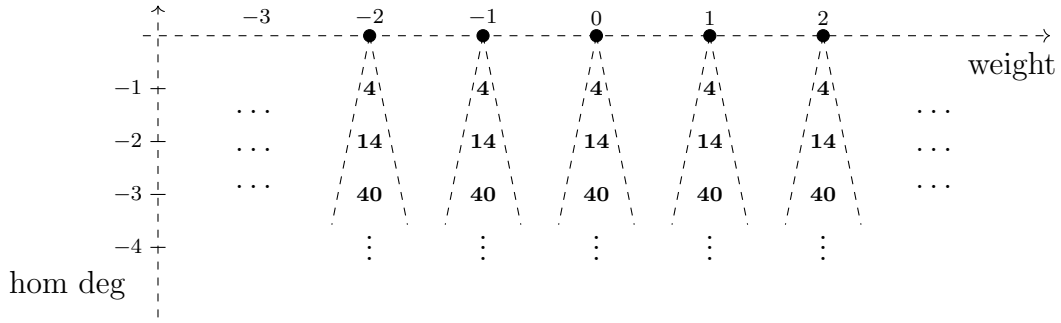

Set $\deg v_0=(0,0)$.
Then the $\Z^2$
grading on $U_q\widehat{\gl}_2$ induces a  $\Z^2$
grading on $\hat L(a,b,c;\kappa)$. We note that a vector of weight $a+2k$ has the first component of degree in  $\hat L(a,b,c;\kappa)$ equal to $k$. 
The grade diagram of $\hat L(a,b,c;\kappa)$ is given on Figure \ref{flat pic}. 
The dimension $d_s$
of the space of degree $(k,-s)$ is independent of $k$ and is given by 
\be
\sum_{s=0}^\infty d_s t^s =\frac{1}{\prod_{i=1}^\infty(1-t^i)^4}= 1 + 4 t + 14 t^2 + 40 t^3 + 105 t^4 + 252 t^5 + \dots \ .
\ee 

Let $\hat L_m(a,b,c;\kappa)$ be the module  $\hat L(a,b,c;\kappa)$ twisted by automorphism $\omega^{m}$. We call the module $\hat L_m(a,b,c;\kappa)$ the slanted relaxed Verma module of slope $m$. That is, 
$\hat L_m(a,b,c;\kappa)=\hat L(a,b,c;\kappa)$ as a vector space and for $v\in \hat L_m(a,b,c;\kappa)$, $g\in U_q\widehat{\gl}_2$, we have $g\cdot v=\omega^{m}(g)v$.

The automorphism $\omega$ changes the homogeneous degree. Thus,  grade diagram of $\hat L_m(a,b,c;\kappa)$ is obtained from that of  $\hat L(a,b,c;\kappa)$ by moving vectors of weight $s$ down by by $sm$. The grade diagram of $\hat L_m(a,b,c;\kappa)$ with $m=2$ is given on Figure \ref{slanted pic}. The dimension of the space down $s$ from the dashed line is still $d_s$.

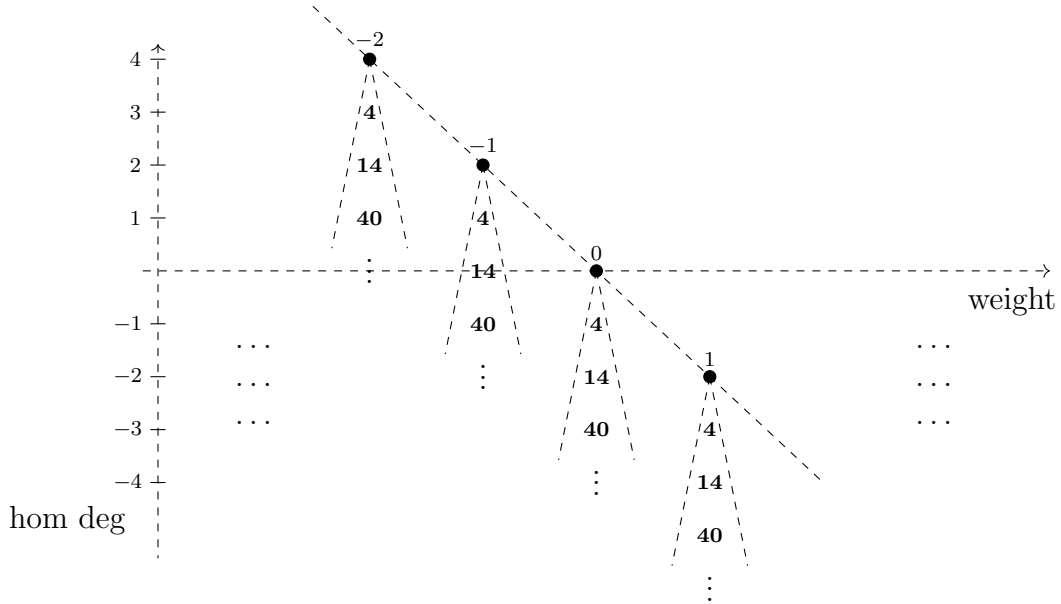
\begin{figure}[ht]
\centering
\begin{tikzpicture}
\draw[dashed, ->] (-6,0) -- (6,0);
\node at (5.5,-0.4) {weight};
\draw[dashed, <-] (-5.8,3) -- (-5.8,-3.8);
\node at (-7.0,-3.3) {hom deg};
\coordinate (a1) at (-4.5,4.2);
\coordinate (a2) at (-3,2.8);
\coordinate (a3) at (-1.5,1.4);
\coordinate (a4) at (0,0);
\coordinate (a5) at (1.5,-1.4);
\coordinate (a6) at (3,-2.8);

\draw[dashed] (-3.75,3.5) -- (a6);

\node[above] at (a2) {\tiny $-2$};
\node[above] at (a3) {\tiny $-1$};
\node[above] at (a4) {\tiny $0$};
\node[above] at (a5) {\tiny $1$};

\node at (-6.2,-0.7) {\tiny $-1$};
\draw (-5.9,-0.7) -- (-5.7,-0.7);
\node at (-6.2,-1.4) {\tiny $-2$};
\draw (-5.9,-1.4) -- (-5.7,-1.4);
\node at (-6.2,-2.1) {\tiny $-3$};
\draw (-5.9,-2.1) -- (-5.7,-2.1);
\node at (-6.2,-2.8) {\tiny $-4$};
\draw (-5.9,-2.8) -- (-5.7,-2.8);

\node at (-6.1,0.7) {\tiny $1$};
\draw (-5.9,0.7) -- (-5.7,0.7);
\node at (-6.1,1.4) {\tiny $2$};
\draw (-5.9,1.4) -- (-5.7,1.4);
\node at (-6.1,2.1) {\tiny $3$};
\draw (-5.9,2.1) -- (-5.7,2.1);
\node at (-6.1,2.8) {\tiny $4$};
\draw (-5.9,2.8) -- (-5.7,2.8);

\draw[fill,black] (a2) circle [radius=0.08];
\draw[fill,black] (a3) circle [radius=0.08];
\draw[fill,black] (a4) circle [radius=0.08];
\draw[fill,black] (a5) circle [radius=0.08];

\draw[dashed] (a2) --++ (0.5,-2.5);
\draw[dashed] (a2) --++ (-0.5,-2.5);
\draw[dashed] (a3) --++ (0.5,-2.5);
\draw[dashed] (a3) --++ (-0.5,-2.5);
\draw[dashed] (a4) --++ (0.5,-2.5);
\draw[dashed] (a4) --++ (-0.5,-2.5);
\draw[dashed] (a5) --++ (0.5,-2.5);
\draw[dashed] (a5) --++ (-0.5,-2.5);


\node at ($(a2)-(0,0.7)$) {\tiny${\bf 4}$};
\node at ($(a2)-(0,1.4)$) {\tiny${\bf 14}$};
\node at ($(a2)-(0,2.1)$) {\tiny${\bf 40}$};
\node at ($(a2)-(0,2.7)$) {$\vdots$};

\node at ($(a3)-(0,0.7)$) {\tiny${\bf 4}$};
\node at ($(a3)-(0,1.4)$) {\tiny${\bf 14}$};
\node at ($(a3)-(0,2.1)$) {\tiny${\bf 40}$};
\node at ($(a3)-(0,2.7)$) {$\vdots$};

\node at ($(a4)-(0,0.7)$) {\tiny${\bf 4}$};
\node at ($(a4)-(0,1.4)$) {\tiny${\bf 14}$};
\node at ($(a4)-(0,2.1)$) {\tiny${\bf 40}$};
\node at ($(a4)-(0,2.7)$) {$\vdots$};

\node at ($(a5)-(0,0.7)$) {\tiny${\bf 4}$};
\node at ($(a5)-(0,1.4)$) {\tiny${\bf 14}$};
\node at ($(a5)-(0,2.1)$) {\tiny${\bf 40}$};
\node at ($(a5)-(0,2.7)$) {$\vdots$};


\node at (4.5,-1) {$\dots$};
\node at (4.5,-1.5) {$\dots$}; 
\node at (4.5,-2) {$\dots$}; 

\node at (-4.5,-1) {$\dots$};
\node at (-4.5,-1.5) {$\dots$}; 
\node at (-4.5,-2) {$\dots$};

\end{tikzpicture}
\caption{Grading in a slope 2 relaxed $U_q\widehat{\gl}_2$ Verma module.}
\label{slanted pic}
\end{figure}

\section{Quantum toroidal $\gl_2$ algebra}\label{sec:tor}

\subsection{The perpendicular realization $\E^\perp$}\label{sec:En}
For $i,j\in\Z/2\Z$ and $r\neq 0$, we set
\begin{align*}
 g_{i,j}(z,w)=\begin{cases}
	      z-q_2w & (i\equiv j),\\
              (z-q_1w)(z-q_3w)& (i\not\equiv j).
	     \end{cases}
\end{align*}

The quantum toroidal algebra  of type $\gl_2$, 
which we denote by $\E^\perp$, is a unital associative
algebra generated by $E^\perp_{i,k},F^\perp_{i,k},H_{i,r}^\perp$ 
and elements $q^{h}$, $C^{\pm 1}$, 
where $i\in\Z/2\Z$, $k\in \Z$, $r\in\Z\backslash\{0\}$, 
$h\in P$.  
We set $K_1^\perp=K=q^{\alpha}$, $K_0^\perp=K^{-1}=q^{-\alpha}$.

We give the defining relations 
in terms of generating series
\begin{align*}
&E_i^\perp(z)=\sum_{k\in\Z}E_{i,k}^\perp z^{-k}, 
\quad
F_i^\perp(z)=\sum_{k\in\Z}F_{i,k}^\perp z^{-k},\\
&K^{\pm,\perp}_i(z)=(K_i^\perp)^{\pm1}\bar{K}^{\pm,\perp}_i(z)\,,
\quad
\bar{K}^{\pm,\perp}_i(z)=
\exp\Bigl(\pm(q-q^{-1})\sum_{r>0}H_{i,\pm r}^\perp z^{\mp r}\Bigr)\,.
\end{align*}
The relations are as follows.

\bigskip

\begin{align*}
&\text{$C$ is central},\quad CC^{-1}=1,\quad
q^{h}q^{h'}=q^{h+h'},\quad
q^0=1\,,
\\
&q^{h}E_i^\perp(z)q^{-h}=
q^{(h,\alpha_i)}E_i^\perp(z)\,,
\quad
q^{h}F_i^\perp(z)q^{-h}=
q^{-(h,\alpha_i)}F_i^\perp(z)\,,
\quad q^hK^{\pm,\perp}_i(z)=K^{\pm,\perp}_i(z)q^h\,,
\end{align*}
\begin{align*}
&K^{\pm,\perp}_i(z)K^{\pm,\perp}_j (w) = K^{\pm,\perp}_j(w)K^{\pm,\perp}_i (z), \\
&\frac{g_{i,j}(C^{-1}z,w)}{g_{i,j}(Cz,w)}
K^{-,\perp}_i(z)K^{+,\perp}_j (w) 
=
\frac{g_{j,i}(w,C^{-1}z)}{g_{j,i}(w,Cz)}
K^{+,\perp}_j(w)K^{-,\perp}_i (z),
\\
&(-1)^{i+j}g_{i,j}(z,w)K_i^{\pm,\perp}(C^{-(1\pm1)/2}z)E_j^\perp(w)+g_{j,i}(w,z)E_j^\perp(w)K_i^{\pm,\perp}(C^{-(1\pm1) /2}z)=0,
\\
&(-1)^{i+j}g_{j,i}(w,z)K_i^{\pm,\perp}(C^{-(1\mp1)/2}z)F_j^\perp(w)+g_{i,j}(z,w)F_j^\perp(w)K_i^{\pm,\perp}(C^{-(1\mp1) /2}z)=0\,,
\end{align*}
\begin{align*}
&[E_i^\perp(z),F_j^\perp(w)]=\frac{\delta_{i,j}}{q-q^{-1}}
(\delta\bigl(C\frac{w}{z}\bigr)K_i^{+,\perp}(w)
-\delta\bigl(C\frac{z}{w}\bigr)K_i^{-,\perp}(z))\,,
\end{align*}
\begin{align*}
&(-1)^{i+j}g_{i,j}(z,w)E_i^\perp(z)E_j^\perp(w)+g_{j,i}(w,z)E_j^\perp(w)E_i^\perp(z)=0, 
\\
&(-1)^{i+j}g_{j,i}(w,z)F_i^\perp(z)F_j^\perp(w)+g_{i,j}(z,w)F_j^\perp(w)F_i^\perp(z)=0,
\end{align*}

\begin{align*}
&\Sym_{z_1,z_2,z_3}[E_i^\perp(z_1),[E_i^\perp(z_2),[E_i^\perp(z_3),E_{i+1}^\perp(w)]_{q^2}]]_{q^{-2}} =0\,,
\\
&\Sym_{z_1,z_2,z_3}[F_i^\perp(z_1),[F_i^\perp(z_2),[F_i^\perp(z_3),F_{i+1}^\perp(w)]_{q^2}]]_{q^{-2}} =0\,.
\end{align*}
Here we use $\delta(z)=\sum_{k\in\Z} z^k$ and
\begin{align*}
&\Sym\ f(x_1,\dots,x_N) =\frac{1}{N!}
\sum_{\pi\in\GS_N} f(x_{\pi(1)},\dots,x_{\pi{(N)}})\,.
\end{align*}

\bigskip

The algebra $\E^\perp$ considered here is obtained from $\E_2$ in
\cite{FJMM2} by setting the second central element $K_0^\perp K_1^\perp$ to $1$, 
and adding the split central element ${\sf t}=q^{\varepsilon_1+\varepsilon_2}$.

\medskip

We note that the algebra $\E^\perp$ is invariant under the exchange of $q_1$ and $q_3$.

Algebra $\E^\perp$ is $(\Z^2\times \Z)$-graded by  
\begin{align*}
\deg E_{i,k}^\perp=(1_i,k), \quad \deg F_{i,k}^\perp=(-1_i,k), \qquad \deg H_{i,r}^\perp=(0,0,r), \qquad  \deg C=\deg q^h=0, 
\end{align*}
where $1_0=(1,0)$ and $1_1=(0,1)$.

For $i=0,1$, there exist automorphisms $\omega_i:\  \E^\perp \to \E^\perp$ sending 
\be
E_j^\perp (z)\mapsto z^{\delta_{ij}} E_j^\perp(z), \qquad F_j^\perp(z)\mapsto z^{-\delta_{ij}} F_j^\perp(z), \qquad  K_i^{\pm,\perp}(z)\mapsto K_i^{\pm,\perp}(z),
\ee
and $q^h\mapsto q^h$, $C\mapsto C$. Note that $\omega_0\,\omega_1=\omega_1\,\omega_0$.
   
We have an embedding $v^\perp:\ U_q\widehat{\sln}_2\to \E^\perp$ such that 
\begin{align*}
&x^+(z)\mapsto E_1^\perp(d^{-1}z)\,,\quad  
x^-(z)\mapsto F_1^\perp(d^{-1}z)\,,\quad  
\phi^\pm(z)\mapsto K^{\pm,\perp}_1(d^{-1}z)\,,
\end{align*}
and $C\mapsto C$. We call the image of $v^\perp$ the vertical $U_q\widehat{\sln}_2$-subalgebra
and denote it by $U^{v,\perp}_q\widehat{\sln}_2\subset \E^\perp$. 

The subalgebra of $\E^\perp$ generated by  $E_1^\perp(z)$, $F_1^\perp(z)$, $K_0^{\pm,\perp}(z)$, $K_1^{\pm,\perp}(z)$,  $q^h$ and $C^{\pm 1}$ is isomorphic to $U_q\widehat{\gl}_2$. We call this subalgebra the vertical $U_q\widehat{\gl}_2$-subalgebra
and denote it by $U^{v,\perp}_q\widehat{\gl}_2\subset \E^\perp$.
 
\begin{prop}[\cite{M2},\cite{FJM1}] Let $u\in \C^{\times}$.
Let $C=q_3$.
There exists a surjecive homomorphism of algebras: $ev_u: \E^\perp \to \widetilde{U}_{q,q_3} \widehat{\gl}_2,$  such that $ev_u\circ v^\perp= id$ and $\deg(ev_u(E_{0,k}^\perp))=(-1,k)$, $\deg(ev_u(F_{0,k}^\perp))=(1,k)$.
\qed
\end{prop}

Note that if $L$ is an $\widetilde{U}_{q,q_3} \widehat{\gl}_2$-module 
and $L_m$ is the module obtained by
twisting $L$ by the automorphism $\omega^{m}$, then 
the evaluation $\E^\perp $-module of $L_m$ is obtained from the evaluation $\E^\perp$-module of $L$ by twisting by the
automorphism $\omega_1^{m}\omega_0^{-m}$.

We call the subalgebra of $\E^\perp$ generated by $E_{i,0}^\perp,F_{i,0}^\perp$, $i=1,2$, the horizontal  $\widehat{\sln}_2$ subalgebra and denote it by $U^{h,\perp}_q\widehat{\sln_2}$. The horizontal 
$\widehat{\sln}_2$ subalgebra is isomorphic to the quotient of
$U_q\widehat{\sln}_2$ 
by the relation $C=1$.

\subsection{Parallel realization of $\E^\perp$}
Algebra $\E^\perp$ has an alternative presentation.
Let $\E$ be the unital associative algebra
generated by  coefficients of the generating series
$$
E_i(z)=\sum_{k\in\Z}E_{i,k}z^{-k}, \quad
F_i(z)=\sum_{k\in\Z}F_{i,k}z^{-k}, \quad
K^{\pm}_i(z)=K_i^{\pm1}\exp\bigl(\pm(q-q^{-1})\sum_{r>0}H_{i,\pm r}z^{\mp r}
\bigr),
$$
where $i\in\Z/2\Z$, and a split central element ${\sf t}$,
which satisfy the following relations. 
\begin{align*}
K^\pm_i(z)K^\pm_j(w) = K^\pm_j(w)K^\pm_i (z), \qquad K^\pm_i(z)K^\mp_j (w) = K^\mp_j(w)K^\pm_i (z),
\\
(-1)^{i+j}g_{i,j}(z,w)K_i^\pm(z)E_j(w)+g_{j,i}(w,z)E_j(w)K_i^\pm(z)=0,
\\
(-1)^{i+j}g_{j,i}(w,z)K_i^\pm(z)F_j(w)+g_{i,j}(z,w)F_j(w)K_i^\pm(z)=0\,,
\end{align*}
\begin{align*}
[E_i(z),F_j(w)]=\frac{\delta_{i,j}}{q-q^{-1}}\delta(z/w)
(K_i^+(w)
-K_i^-(z)),\\
(-1)^{i+j}g_{i,j}(z,w)E_i(z)E_j(w)+g_{j,i}(w,z)E_j(w)E_i(z)=0, \\
(-1)^{i+j}g_{j,i}(w,z)F_i(z)F_j(w)+g_{i,j}(z,w)F_j(w)F_i(z)=0,\\
\on{Sym}_{z_1,z_2,z_3}[E_i(z_1),[E_i(z_2),[E_i(z_3),E_{i+1}(w)]_{q^2}]]_{q^{-2}}=0, \\
\on{Sym}_{z_1,z_2,z_3}[F_i(z_1),[F_i(z_2),[F_i(z_3),F_{i+1}(w)]_{q^2}]]_{q^{-2}}=0.
\end{align*}
There exists a Drinfeld coproduct $\Delta$ in $\E$ given by
\begin{align}
&\Delta(K_i^{\pm}(z))=K_i^{\pm}(z)\otimes K_i^{\pm}(z), \notag \\
&\Delta(E_i(z))=E_i(z)\otimes 1+ K_i^-(z)\otimes E_i(z),\label{coproduct}\\
&\Delta(F_i(z))=F_i(z)\otimes  K_i^+(z)+ 1\otimes F_i^{\pm}(z),\notag 
\end{align}
and $\Delta({\sf t})={\sf t}\otimes {\sf t}$.

There exists a map  $h:\ U_q\widehat{\sln}_2\to \E$ such that 
\begin{align*}
&x^+(z)\mapsto E_1(d^{-1}z)\,,\quad  
x^-(z)\mapsto F_1(d^{-1}z)\,,\quad  
\phi^\pm(z)\mapsto K^{\pm}_1(d^{-1}z)\,, \qquad C\mapsto 1.
\end{align*}
We denote the image of $h$ by $U^{h}_q\widehat{\sln}_2\subset \E$. 
The algebra $U^{h}_q\widehat{\sln}_2$ is isomorphic to the quantum affine algebra $U_{q,1}\widehat{\sln}_2$ 
with $C=1$. 

We denote the subalgebra of $\E$ generated by $E_{i,0},F_{i,0}$, $i=1,2$,  by $U^{v}_q\widehat{\sln_2}$. 

Define a $\Z^3$ grading in $\E$ by
\begin{align}
&\deg E_{1,k}=(-k,-k+1,0), \qquad &\deg E_{0,k}=(-k,-k-1,1), \notag\\
&\deg F_{1,k}=(-k,-k-1,0), \qquad &\deg F_{0,k}=(-k,-k+1,-1),\hspace{10pt}\label{grading} \\
&\deg H_{1,r}=(-r,-r,0), \qquad &\deg H_{0,r}=(-r,-r,0),\hspace{30pt} \notag
\end{align}
and $\deg {\sf t}=(0,0,0)$. 

It is easy to check it is indeed a grading. The definition of this grading is chosen to make the isomorphism in Proposition \ref{isom prop} below graded, cf. Lemma 2.4 in \cite{FJMM2}.

The algebras $\E^\perp$ and $\E$ are, in fact, the same.

\begin{prop}[\cite{M1}]\label{isom prop} There exists a $\Z^3$-graded isomorphism of algebras $\theta:\ \E^\perp \to \E$ such that $\theta(C)=K_0K_1$,  $\theta(U^{v,\perp}_q\widehat{\sln_2})=U^{v}_q\widehat{\sln_2}$, $\theta(U^{h,\perp}_q\widehat{\sln_2})=U^{h}_q\widehat{\sln_2}$, and $\theta (q^h)=q^h$, $h\in P$.
\qed
\end{prop}

The commutative currents $K_i(z)$ are called Cartan currents.

\medskip 

For $a\in\C^\times$, the shift of spectral parameter automorphism $\hat\tau_a:\ \E\to \E$ maps 
\begin{align}\label{toroidal shift}
E_i(z)\mapsto E_i(z/a), \qquad 
F_i(z)\mapsto F_i(z/a), \qquad 
K_i^\pm(z)\mapsto K_i^\pm(z/a),
\end{align}
and $\sf t  \to \sf t$.  The automorphism 
$\hat\tau_a$ preserves the degree. Note that the automorphism $\hat \tau_a$ preserves
subalgebras $U_q^h\widehat{\sln_2}$ and $U_q^v\widehat{\sln_2}$. The restriction of $\hat \tau_a$ to $U_q^h\widehat{\sln_2}$ is the automorphism $\tau_a$, see \eqref{shift}, and the restriction of $\hat \tau_a$ to $U_q^{v}\widehat{\sln_2}$ is the identity map.

\medskip

Let $V$ be an $\E$-module. 

For $\kappa \in\C^{\times}$, we
say $V$ has level $\kappa$ if the central element $K_0K_1$ acts in $V$ 
as the scalar $\kappa$.

A vector $v\in V$ is called a vector of $\ell$-weight  $\phi(z)=(\phi^0(z),\phi^1(z))$
 if $K_i^\pm(z)\,v=\phi^i(z)\, v$, $i=0,1$.

A module $V$ is called a highest $\ell$-weight module of highest $\ell$-weight $\phi(z)$ if it is generated by a vector $v$ of $\ell$-weight 
$\phi(z)$ 
satisfying $E_i(z)\,v=0$, $i=0,1$.

\begin{prop}{\cite{FJM1}}
Let  $L_{a,b}$ be
the irreducible highest weight
$U_q\widehat{\gl_2}$-module of highest weight $a,b$, that is generated by 
a non-zero vector $v$ such that
$K_0\,v=av$, $K_1\,v=bv$, $E_{i,0}\,v=0$, $i=0,1$. Set $q_3=ab$. Then
the corresponding evaluation $\E$-module
has highest $\ell$-weight 
\beq\label{evaluation weight}
\phi(z)
=\big(\frac{az-u}{z-au},\  \frac{bz-u}{z-bu}\big).  
\eeq
Here the parameter $u$ depends on the choice of the evaluation map and can be made arbitrary by twisting by automorphism \eqref{toroidal shift}. \qed
\end{prop}

In all modules considered in this paper, the series $\phi_i(z)$ will be rational functions regular at zero and infinity, satisfying $\phi_i(0)\phi_i(\infty)=1$, cf. Theorem 2.3 in \cite{FJMM1}. We call such rational functions balanced. The set of all  pairs of rational functions $\phi(z)$
with balanced coordinates is naturally a group under coordinate-wise multiplication.

A module $V$ is called tame if 
$V$ has a basis of $\ell$-weight vectors with distinct $\ell$-weights.

\subsection{An explicit construction of tame $\E$-modules} \label{sec:combinatorial}
If a module over quantum affine algebra
is tame then the matrix coefficients of $E_i(z)$ and $F_i(z)$ are delta functions multiplied by some constants. We give an abstract construction which under some assumption defines the action in a completely combinatorial way.

A box $\Box$ is a set which has two elements: color of the box 
$c(\Box)\in\{0,1\}$ and coordinate
of the box 
$(x(\Box),y(\Box),z(\Box))\in\C^3$. We often ignore the color and simply write $\Box=(x,y,z)$.

The position of the box is then given by $p(\Box)=q_1^{-x(\Box)}q_2^{-y(\Box)}q_3^{-z(\Box)}$. We assume that the set of all boxes is totally ordered. 

We do not allow boxes with the same positions and different colors.

If all three coordinates of two boxes differ by integers then we assume that the ordering is lexicographic according to $y>z>x$. Namely, two
such boxes are ordered as $\Box_1>\Box_2$ 
if and only if  $y(\Box_1)>y(\Box_2)$ or if $y(\Box_1)=y(\Box_2)$ and $z(\Box_1)>z(\Box_2)$, or if 
$y(\Box_1)=y(\Box_2)$ and $z(\Box_1)=z(\Box_2)$ and $x(\Box_1)>x(\Box_2)$.  One can take any ordering satisfying this assumption.

If the $x$ and $z$ coordinates of two boxes differ by integers then we assume $c(\Box_1)-c(\Box_2)\equiv x(\Box_1)-x(\Box_2)+z(\Box_1)-z(\Box_2)$ modulo $2$. In other words, colors alternate in $x$- and $z$-directions and do not change in $y$-direction. 

The set of all boxes is divided into
two disjoint subsets of positive and negative boxes.

A state $\la=\la^+\sqcup \la^-$ is a disjoint union 
of two finite sets of
boxes without repetitions
such that all boxes in $\la^+$ are positive, and all boxes in $\la^-$ are negative.

The state $\emptyset$ is called the reference state. 
We think of a state $\la$ as having
positive boxes in $\la^+$ added to the reference state and negative boxes in $\la^-$  removed from the reference state. So, one can think that the reference state consists of all negative boxes.

Let $S$ be a collection of states including the reference state, and let $\Psi(z)=(\Psi^0(z),\Psi^1(z))$ be a pair of balanced rational functions. This is our data for constructing an $\E$-module
under
the assumption to be stated below.

For $\la\in S$, a box $\Box$ is called convex if either $\Box\in \la^+$ and the removal of this box produces a state,
$((\la^+\setminus \Box)\sqcup \la^-) \in S$, or 
$\Box\not \in \la^-$ is a negative box and the addition of this box to $\la^-$   produces a state,
$(\la^+\sqcup (\la^-\sqcup\Box)) \in S$. In both cases we say that the new state is obtained from $\la$ by removing the box $\Box$ and we denote it by $\la\setminus \Box$.
We denote the set of all convex boxes of color $i$ for a state $\la$ by $CV_i(\la)$. 

For $\la\in S$, a box $\Box$ is called concave if either $\Box\not\in \la^+$ is a positive box and the addition
of this box to $\la^+$ produces a state,
$((\la^+\sqcup \Box)\sqcup \la^-) \in S$, or 
$\Box \in \la^-$ and the removal of this box produces a state,
$(\la^+\sqcup (\la^-\setminus \Box)) \in S$. In both cases we say that the new state is obtained from $\la$ by adding the box $\Box$ and we denote it by $\la\sqcup \Box$.
We denote the set of all concave boxes of color $i$ for a state $\la$ by $CC_i(\la)$.

A set 
$S$ of states
is connected if any non-empty state 
$\la$ can be connected to the reference state by a sequence of removal of boxes.

Next we define the $\ell$-weight of a state.
For $a\in\C^\times$, introduce the notation for pairs of balanced rational functions
\begin{equation*}
 0_a=(\frac{qz-a}{z-aq},\ 1), \qquad 
 1_a=(1,\ \frac{qz-a}{z-aq}).
\end{equation*}
These are quantum affine
analogs of $\widehat{\sln}_2$ fundamental weights. The 
quantum  affine 
analogs of $\widehat{\sln}_2$  roots are given by 
\begin{align*}
A_{0,a}(z)&= \Big(
-\frac{g_{0,0}(a,z)}{g_{0,0}(z,a)}, \frac{g_{0,1}(a,z)}{g_{1,0}(z,a)}
\Big)=\Big(\frac{q_2 z-a}{z-q_2 a},   \frac{(q_1 z-a)(q_3 z-a)}{(z-q_1 a)(z-q_3 a)}\Big)=0_{q^{-1}a}  0_{qa} 1^{-1}_{q_1 qa} 1^{-1}_{q_3 qa}, \\
A_{1,a}(z)&= \Big(\frac{g_{1,0}(a,z)}{g_{0,1}(z,a)}, -\frac{g_{1,1}(a,z)}{g_{1,1}(z,a)}
\Big)=
\Big(\frac{(q_1 z-a)(q_3 z-a)}{(z-q_1 a)(z-q_3 a)},
\frac{q_2 z-a}  {z-q_2 a}
\Big)
=1_{q^{-1}a}  1_{qa} 0^{-1}_{q_1 qa} 0^{-1}_{q_3 qa}.
\end{align*}

Note the complete symmetries between $0 \leftrightarrow 1$ and between $q_1\leftrightarrow q_3$.

Note the identities
\begin{align}\label{A symmetries)}
A_{i,a}(b)=A_{i,1}(b/a), \qquad A_{i,a}(b)= A_{i,b}^{-1}(a), \qquad A_{0,a}(a)=(-1,1),\qquad A_{1,a}(a)=(1,-1).
\end{align}

We define the  $\ell$-weight of 
a box $\Box$ to be $A^{-1}_{c(\Box), p(\Box)}(z)$.

We define the $\ell$-weight of the reference state to be $\Psi(z)$,  that is $\phi_\emptyset(z)=\Psi(z)$.

We define the $\ell$-weight of a state $\la\in S$ by 
\be
\phi_\la(z)=(\phi^0_\la(z),\phi^1_\la(z))=
\Psi(z) \prod_{\Box\in\la^+} A^{-1}_{c(\Box), p(\Box)}(z) \prod_{\Box\in\la^-} A_{c(\Box),  p(\Box)}(z).
\ee 

The following lemma is a trivial but important consequence of the definition o $\phi_\la(z)$.
\begin{lem}\label{trivial lem} Let  $\la\in S$ be a state and let $\Box$ be a box.  Then 
\begin{align*}
&\phi_{\la\sqcup \Box}(z)= A^{-1}_{c(\Box),p(\Box)}(z) \phi_{\la}(z)
\quad \text{if $\Box$ is concave for $\lambda$},
\\ 
&\phi_{\la\setminus \Box}(z)= A_{c(\Box),p(\Box)}(z) \phi_{\la}(z)
\quad\, \text{if $\Box$ is\    convex\   \   for $\lambda$}\,.
\end{align*}
\qed
\end{lem}

Given a box $\Box$ we break the $\ell$-weight of a state into a product  by collecting boxes before and after  $\Box$:
\begin{align*}
\phi_{\la,<\Box}(z)&=\Psi (z)\prod_{\Box_1\in\la^+, \,\Box_1<\Box } A^{-1}_{c(\Box_1), p(\Box_1)}(z) \prod_{\Box_1\in\la^-, \,\Box_1<\Box } A_{c(\Box_1), p(\Box_1)}(z), \\
\phi_{\la,>\Box}(z)&=\prod_{\Box_1\in\la^+, \,\Box_1>\Box } A^{-1}_{c(\Box_1), p(\Box_1)}(z) \prod_{\Box_1\in\la^-, \,\Box_1>\Box } A_{c(\Box_1), p(\Box_1)}(z).
\end{align*}

We have 
\begin{align}\label{phi lambda|}
\phi_\la(z)
=\begin{cases} 
\phi_{\la,<\Box}(z) \,\phi_{\la,>\Box}(z), & {\rm if }\  \Box\not \in\la, \\
\phi_{\la,<\Box}(z) \,\phi_{\la,>\Box}(z)\,
A^{-1}_{c(\Box),p(\Box)}(z)\,,
& {\rm if }\ \Box\in \la^+,\\
\phi_{\la,<\Box}(z) \,\phi_{\la,>\Box}(z)\,
A_{c(\Box),p(\Box)}(z)\,,
& {\rm if }\ 
\Box\in \la^-.
\end{cases}
\end{align}
\medskip 

We make the following assumptions: 
\begin{itemize}
    \item[(A1)] The $\ell$-weights of all states are distinct.
    \item[(A2)] For each $\la\in S$ the components $\phi_\la^i(z)$ of $\ell$-weight $\phi(\la)$ are balanced rational functions with only simple poles.
    \item[(A3)] For $i=0,1$, the map $\Box\mapsto p(\Box)$ gives a bijection from $CC_i(\la)\sqcup CV_i(\la)$ to the set of poles of $\phi^i_{\la}(z)$.
    \item[(A4)] Let $\Box\in CC_i(\la)\sqcup CV_i(\la)$.  If $\Box$ is positive, then $\phi^i_{\la,>\Box}(z)$ is regular at $z=p(\Box)$. If $\Box$ is a negative, then $\phi^i_{\la,<\Box}(z)$ is regular at $z=p(\Box)$. 
    \item[(A5)] Let a box $(x,y,z)$ be concave in $\la$. If a box with coordinates $(x',y',z')$ is concave in $\la\sqcup (x,y,z)$ but not in  $\la$ then $(x',y',z')\in\{(x,y,z+1), (x,y+1,z),(x+1,y,z)\}$. If a box with coordinates $(x',y',z')$ is convex in $\la$ but not in  $\la\sqcup (x,y,z)$ then $(x',y',z')\in\{(x,y,z), (x,y,z-1), (x,y-1,z),(x-1,y,z)\}$.
\end{itemize}
\medskip

In particular, the assumptions (A2), (A3) imply that for each state $\lambda\in S$, 
the sets of convex and concave boxes $CV_i(\la)$, $CC_i(\la)$
are finite.  By (A2) and (A4), the positions of all these boxes are distinct.

Using assumptions (A3),  (A4), one can also deduce our previous assumption on coloring but we do not pursue this point.

Note that boxes $(x,y,z)$ and $(x+t,y+t,z+t)$ have the same position and color. For all our purposes, such boxes are indistinguishable. Assumption (A5) takes away this freedom and ensures that the boxes in a state are aligned together. For the proofs below (A5) could be omitted.

Now, we define the $\E$-module $V(S,\Psi)$.

Let $V=\oplus_{\la\in S}\C \ket{\la}$ 
be the $\C$-vector space with a fixed basis labeled by states in $S$.

Set 
\begin{align}
&K_i^\pm(z)\ket{\la} = \phi^i_\la(z)\,\ket{\la}, \notag\\
&F_i(z)\ket{\la}= \sum_{\Box\in CC_i(\la)} a_{\la,\Box}^i\, \delta(z/p(\Box))\, \ket{\la\sqcup \Box},\label{action}
\\
&E_i(z)\ket{\la}= \sum_{\Box\in CV_i(\la)} b_{\la,\Box}^i\,\delta(z/p(\Box)) \,\ket{\la\setminus \Box},\notag
\end{align}
where the coefficients
are given by
\begin{align}  
a_{\la,\Box}^i&= \begin{cases} \phi_{\lambda,>\Box}^i(p(\Box)), 
& {\rm if\ }\Box {\rm \ is\  positive,}\\
\frac{-1}{q-q^{-1}}\Res_{z=p(\Box)}
\phi_{\lambda,>\Box}^i(z)\frac{dz}{z}, 
& {\rm if\ }\Box {\rm \ is\  negative,}
\end{cases}\label{constants 1}
\\
b_{\la,\Box}^i&= \begin{cases} \frac{1}{q-q^{-1}}
\Res_{z=p(\Box)} 
\phi_{\lambda,< \Box}^i(z)\frac{dz}{z}, 
& {\rm if\ }\Box {\rm \ is\  positive},\\
\phi_{\lambda,< \Box}^i(p(\Box)), 
& {\rm if\ }\Box {\rm \ is\  negative}.
\end{cases}\label{constants 2}
\end{align}
The element ${\sf t}$ can be set to an arbitrary non-zero number. For brevity, we set ${\sf t}=1$.

Note that coefficients $a_{\la,\Box}$ and $b_{\la,\Box}$ are well defined by (A4) and non-zero by (A3), (A2). Note that by (A2), the $\phi^i_\la(z)$, $i=0,1$, are balanced, so $K_i=(K_i^-(0))^{-1}=K_i^+(\infty)$ are well-defined. 

The statement that in 
tame modules the matrix 
coefficients of $F_i(z),E_i(z)$
must be delta functions at poles of $\ell$-weights is an easy general fact. 
However, in general it is not clear how to fix the 
coefficients,
see Section 5.3 in \cite{FJM2}. 
Here we regard our module as some kind of tensor product, 
therefore the choice of coefficients
is dictated by the coproduct \eqref{coproduct}.
Now we check that it works.

\begin{thm}\label{main thm}
Let $S$ be a connected set of states and $\Psi(z)$ 
a pair of balanced functions such that the assumptions (A1)--(A5) hold. 
Then formulas \eqref{action} make 
$V=\oplus_{\la\in S}\C \ket{\la}$ 
into a tame irreducible $\E$-module. 
\end{thm}
We prove Theorem \ref{main thm} in the Subsection \ref{proof subsection}.

We denote by $V(S,\Psi)$ the $\E$-module $V$
constructed by the use of Theorem \ref{main thm} 
and call it a combinatorial module. 

\medskip

Define the degree of a box of color $0$ to be $(1,0)$ and the degree of a box of color $1$ to be $(0,1)$. The degree of a collection of boxes is the sum of degrees of the boxes.
For a state $\la=\lambda^+\sqcup\lambda^-\in S$, define the two component degree of $\la$ by
\be
\deg(\la)=\deg(\la^+)-\deg(\la^-)=(\deg_0(\la),\deg_1(\la)).
\ee
Note that the degree is normalized by the condition that the degree of the reference state is $(0,0)$. 

Define the character of a combinatorial module $V(S,\Psi)$ as a generating series of degrees of the states:
\be
\chi_{V(S,\Psi)}(z_0,z_1)=\sum_{\la\in S} z_0^{\deg_0(\la)}z_1^{\deg_1(\la)}.
\ee 

\medskip

The results of this section, including Theorem \ref{main thm}, also hold for quantum toroidal algebras of type
$\gl_n$ with arbitrary $n$ with obvious changes in the coloring.

\subsection{The proof of Theorem \ref{main thm}}\label{proof subsection}
 All the relations are checked directly.  We give some details.

Since $K_i^\pm (z)$ are diagonal, $K_i, H_{i,r}$ all commute. 

\medskip

Consider the $K_0F_0$ relation acting on a state $\la$.  Operators $K_0(z)F_0(w)$,  $F_0(w)K_0(z)$ can add a positive box of color $0$ or remove a negative box of color $0$. Let $\Box$ be a concave box of $\la$ of color $0$ with position $p$. Then the corresponding matrix coefficient is 
\begin{align*}
\bra{\la\sqcup \Box}&\Big((w-q_2z) K_0(z)F_0(w)-(z-q_2w)F_0(w)K_0(z)\Big)\ket{\la}\\ &=a_{\la,\Box}^0\,\delta(w/p)\big((w-q_2z)\phi^0_{\la\sqcup\Box}(z)-(z-q_2w)\phi^0_\la(z)\big),
\end{align*}
This matrix coefficient is zero since
\begin{align*}
 \phi^0_{\la\sqcup\Box}(z)=\phi^0_{\la}(z) (A^{-1}_{0,p}(z))_0=  \phi^0_{\la}(z)\frac{z-q_2p}{q_2z-p}.
\end{align*}
The other cases of $KF$ and $KE$ relations are similar.

\medskip

Consider the $F_0F_1$ relation acting on a state $\la$. Operators $F_0(w)F_1(z)$,  $F_1(z)F_0(w)$ can add two positive boxes of colors $0$ and $1$ or remove two negative boxed of colors $0$ and $1$ or remove one negative box of some color and add one positive box of the different color. Consider two boxes $\Box_0<\Box_1$  with positions $p_0$ and $p_1$. Suppose both boxes are concave in a state $\la$. 

Note that $p_0/p_1\not\in\{q_1^{\pm1},q_3^{\pm1}\}$. Indeed, if, for example, $p_0=p_1q_1$ then by Lemma \ref{trivial lem}, the $\ell$-weight of the state $\la\sqcup \Box_0$ has a double pole at $z=p_1$, which contradicts (A2). It follows that $\Box_1$ is concave in $\la\sqcup \Box_0$ and $\Box_0$ is concave in $\la\sqcup \Box_1$.

Then we compute the matrix coefficient, 
\begin{align*}
&\bra{\la\sqcup \Box_0\sqcup\Box_1}\Big((w-q_1z)(w-q_3z) F_1(z)F_0(w)-(z-q_1w)(z-q_3w)F_0(w)F_1(z)\Big)\ket{\la}\\ &=
    \delta(z/p_1)\delta(w/p_0)
    \big( a^1_{\la\sqcup\Box_0,\Box_1}a^0_{\la,\Box_0}(w-q_1z)(w-q_3z) - a^0_{\la\sqcup\Box_1,\Box_0} a^1_{\la,\Box_1} (z-q_1w)(z-q_3w) \big).  
\end{align*}
Since $\Box_0<\Box_1$, we have 
\begin{align*}   a^1_{\la\sqcup\Box_0,\Box_1}&=a^1_{\la,\Box_1},\\ a^0_{\la\sqcup\Box_1,\Box_0}&=a^0_{\la,\Box_0}\frac{(p_0-q_1p_1)(p_0-q_3p_1)}{(q_1p_0-p_1)(q_3p_0-p_1)},
\end{align*}
and the matrix coefficient is zero. The case $\Box_0>\Box_1$ is similar.

 Suppose $\Box_0$ is concave in $\la$, and $\Box_1$ is concave in $\la\sqcup\Box_0$ but not in $\la$. By Lemma \ref{trivial lem} and by (A2), $p_1=q_1^{-1}p_0$ or $p_1=q_3^{-1}p_0$.
 Then $F_0(w)F_1(z)\ket{\la}=0$, and
 \begin{align*}
&\bra{\la\sqcup \Box_0\sqcup\Box_1}\Big((w-q_1z)(w-q_3z) F_1(z)F_0(w)\Big)\ket{\la}\\&=a^0_{\la,p_0}a^1_{\la\sqcup \Box_0,p_1}\delta(z/p_1)\delta(w/p_0)(w-q_1z)(w-q_3z) =0.
\end{align*}
The $F_0F_1$ relations is proved.  

The $F_iF_i$ and $E_iE_j$ quadratic relations are similar.

\medskip
Consider the action of $E_0(w)F_1(z)-F_1(z)E_0(w)$ on a state $\ket{\la}$. Let $\Box_1$, be a concave box of color $1$ and position $p_1$. Let $\Box_0>\Box_1$ be a concave box of color $0$ and position $p_0$. Lemma \ref{trivial lem}  and assumption (A2) imply that  the following three statements are equivalent: (i) $\Box_0\not\in CV(\la\sqcup\Box_1)$;  (ii) $\Box_1\not\in CC(\la\setminus \Box_0)$; (iii) $p_0/p_1\in\{q_1,q_3\}$. If these statements hold, then 
\be
\bra{\la\sqcup \Box_1\setminus \Box_0}\,E_0(w)F_1(z)\ket{\la}=\bra{\la\sqcup \Box_1\setminus \Box_0}\,F_1(z)E_0(w) \ket{\la}=0
\ee 
Otherwise,
\begin{align*}
&\bra{\la\sqcup \Box_1\setminus \Box_0} \big(E_0(w)F_1(z)-F_1(z)E_0(w)\big) \ket {\la}\\
&=\delta(w/p_0)\delta(z/p_1)\big(a^1_{\la,\Box_1}b^0_{\la\sqcup\Box_1,\Box_0} - a^1_{\la\setminus \Box_0,\Box_1}b^0_{\la,\Box_0}(z)\big) \\
&=\delta(w/p_0)\delta(z/p_1)a^1_{\la\setminus \Box_0,\Box_1}b^0_{\la,\Box_0} \big((A^{-1}_{1,p_1}(p_0))_0 (A^{-1}_{0,p_0}(p_1))_0 -1 \big)=0,
\end{align*}
where in the last step we used the second equation in \eqref{A symmetries)}.  


It follows that  $E_0(w)F_1(z)=F_1(z)E_0(w)$.

\medskip

Consider $[E_0(z),F_0(w)]$ applied to a state $\ket{\la}$. Similarly to the previous computation, non-diagonal matrix coefficients vanish. The diagonal matrix coefficient is a sum of terms corresponding to adding and removing a concave box of color $0$ or removing and adding a convex box of color $0$. We have
\begin{align*}
&\bra{\lambda}E_0(z)F_0(w)\ket{\lambda} 
=\sum_{\Box\in CC_0(\lambda)}\delta(w/p(\Box))\delta(z/p(\Box))\,
a_{\lambda,\Box}b_{\lambda\sqcup\Box,\Box}\,,\\
&\bra{\lambda}F_0(w)E_0(z)\ket{\lambda} 
=\sum_{\Box\in CV_0(\lambda)}\delta(w/p(\Box))\delta(z/p(\Box))\,
a_{\lambda\setminus\Box,\Box}b_{\lambda,\Box}\,.
\end{align*}
In view of \eqref{phi lambda|} and $(A_{0,p_0}(p_0))=-1$, we have 
\begin{align*}
&a_{\lambda,\Box}b_{\lambda\sqcup\Box,\Box}=\frac{1}{q-q^{-1}}
\mathop{\res}_{z=p(\Box)}\phi^0_{\la}(z)\frac{dz}{z}
\quad \text{for $\Box\in CC_0(\lambda)$}\,,\\
&a_{\lambda\setminus\Box,\Box}b_{\lambda,\Box}=-\frac{1}{q-q^{-1}}
\mathop{\res}_{z=p(\Box)}\phi^0_{\la}(z)\frac{dz}{z}
\quad \text{for $\Box\in CV_0(\lambda)$}\,.
\end{align*}
By (A3), (A4), the positions of concave and convex boxes of color $0$ are exactly poles of $\phi^0_\la(z)$ and these poles are simple. Therefore, 
\begin{align*}
\bra{\lambda}[E_0(z),F_0(w)]\ket{\lambda} 
=\delta(z/w)\frac{1}{q-q^{-1}}
\sum_{p_0}
\delta(z/p_0)\mathop{\res}_{z=p_0}\phi^0_{\la}(z)\frac{dz}{z}\,, 
\end{align*}
the sum ranging over all poles $p_0$ of $\phi^0_\lambda(z)$. 
Since $\phi^0_\lambda(z)$ a balanced rational function with simple poles, the sum in the right hand side equals to the difference of expansions of $\phi^0_\lambda(z)$ at $z=\infty$ and $z=0$.
This proves the $E_0F_0$ relation.

\medskip

Finally, we discuss Serre relations. Again, there are many cases with positive boxes, negative boxes, ordering of the boxes, choice of colors, choice of $E$ or $F$. We give details for the case of three $F_1$, one $F_0$, boxes  $\Box_i$, $i=1,2,3$, of color $1$ and positions $p_i$ and box $\Box_0$ of color $0$ and position $p_0$, such that $\Box_1>\Box_2>\Box_3>\Box_0$. Suppose all these four boxes are concave in $\lambda$. Then, as before, for all permutations $\sigma\in S_4$, $\Box_{\sigma_i}$ is concave in  $\la\sqcup_{j=1}^{i-1}\Box_{\sigma_j}$
Therefore the matrix coefficient between $\ket{\la\sqcup\Box_0\sqcup_{i=1}^3\Box_i}$ and $\ket{\la}$ for the Serre relation contains $6\times 24$ summands. The first factor $6$ corresponds to choice of $\delta$ functions: for each permutation $\sigma=(\sigma_1,\sigma_2,\sigma_3)\in S_3$ of three elements, we have 24 terms containing $\delta(w/p_0)\prod_{i=1}^3 \delta(z_i/p_{\sigma_i})$. We claim that the sum of each group of 24 terms is zero. Indeed, denote the coordinates of the affine roots:
\begin{align*}
    h(z_1,z_2)=-\frac{z_1-q_2z_2}{z_2-q_2z_1}=(A_{0,z_1}(z_2))_0 , \qquad  k(z,w)=\frac{(z-q_1w)(z-q_3w)}{(w-q_1z)(w-q_3z)}=(A_{0,z}(w))_1 .
\end{align*}
Then the matrix coefficient of operator $[F_1(z_1)[F_1(z_2)[F_1(z_3),F_0(w)]_{q_2}]]_{q_2^{-1}}$ is  given by
\begin{align*}
\delta(w/p_0)\prod_{i=1}^3 \delta(z_i/p_{\sigma_i})a^0_{\la,\Box_0} a^1_{\la,\Box_{\sigma_3}}a^1_{\la\sqcup\Box_{\sigma_3},\Box_{\sigma_2}}a^1_{\la\sqcup\Box_{\sigma_3}\sqcup\Box_{\sigma_2},\Box_{\sigma_1}} S(z_1,z_2,z_3), 
\end{align*}
where
\begin{align*}
S(z_1,z_2,z_3)=(1-q_2 k(z_3,w))(1-k(z_3,w)k(z_2,w))(1-q_2^{-1}k(z_3,w)k(z_2,w)k(z_1,w)).
\end{align*}
After symmetrization with respect  to $z_1,z_2,z_3$, we obtain zero
due to the explicit identity of rational functions
\begin{align}\label{Serre identity}
S(z_1, z_2, z_3)& + S(z_2, z_1, z_3) h(z_1, z_2) +
   S(z_1, z_3, z_2) h(z_2, z_3) + S(z_2, z_3, z_1) h(z_1, z_2) h(z_1, z_3)\notag \\& +
   S(z_3, z_1, z_2) h(z_1, z_3) h(z_2, z_3)+
   S(z_3, z_2, z_3) h(z_1, z_2) h(z_1, z_3) h(z_2, z_3)=0
\end{align} 
(where one should set $z_i=p_{\sigma_i}$ and $w=p_0$).

This is the generic case. We have a number of cases of codimension one or two when some terms are missing.
Suppose, for example, that only $\Box_2$ and $\Box_3$ are convex in $\la$. By (A5) it happens when $p_0=q_1^{-1}p_2$ (or $p_0=q_3^{-1}p_2)$ and $p_0=q_1p_1$ (or $p_0=q_3p_1$). In particular, $\Box_1>\Box_0>\Box_2$.  Then, out of 24 terms only 4 survive which also sum up to zero (here we treat the case $\Box_2>\Box_3$):
\begin{align*} 
1+h(p_3,p_1)-h(p_3,p_1)k(p_3,p_0)-h(p_3,p_1)k(p_3,p_0)h(p_3,p_2)=0.
\end{align*} 
This identity can be checked directly or deduced from \eqref{Serre identity} by taking limit $p_2\to q_1^{-1}p_0$, $p_1\to q_1p_0$. In that limit $k(p_2,p_0)$ has a pole and $k(p_1,p_0)$ has a zero,  while all other values are regular and non-zero. So, the four terms are exactly those which have a pole and therefore their sum vanishes in the limit.

All other cases are checked similarly.

\medskip

The module $V$ is tame by (A1). It is irreducible, since $S$ is connected.

\section{Examples of combinatorial $\E$-modules.}\label{sec:examples}
It seems that Theorem \ref{main thm} can be applied to many combinatorial sets $S$. We have no classification of all possibilities. Here we give a number of examples starting with known ones previously constructed by different methods and then proceed to new ones. In each case we describe the set $S$ and the initial weight $\Psi(z)$ 
and then check the assumptions of Theorem \ref{main thm}. 

\subsection{Vector representations}
We define a combinatorial module $U_0$ associated with 
the data $S(U_0),\Psi(U_0)$ given as follows.

 Let $\Psi(U_0)=0_{q^{-1}}1_{qq_1}^{-1}$. Let the set of boxes consist of boxes with coordinates $(i,0,0)$, and color $i$ mod 2, $i\in\Z$. The box with coordinates $(i,0,0)$ is negative if and only if $i$ is negative.

The set of states is $S(U_0)=\Z$. For $k\geq 0$, the state $k$ consists of $k$ positive boxes with coordinates $(0,0,0),(1,0,0),\dots,(k-1,0,0)$.   For $k< 0$, the state $k$ consists of $|k|$ negative boxes with coordinates $(-1,0,0),(-2,0,0),\dots,(k,0,0)$. The state $0$ 
is the reference state.

\begin{prop}
The space $U_0=V(S(U_0),\Psi(U_0))$ with action \eqref{action} is an irreducible tame $\E$-module of level $1$ called the vector module of color $0$.
\end{prop} 
\begin{proof}

Indeed, by induction on $k$, clearly the $\ell$-weight of the state $k$ is
\begin{align}\label{vector monomial}
\phi_k
=\begin{cases}  
0_{q^{-1}{q_1^{-k}}} 1_{qq_1^{-k+1}}^{-1},  & {\rm if }\  
k {\rm \ is\  even}, \\
0_{q q_1^{-k+1}}^{-1} 1_{q^{-1}q_1^{-k}},  & {\rm if }\ 
k {\rm\ is\  odd}.
\end{cases}
\end{align}
At the same time there is a convex box $(k-1,0,0)$ and a concave box $(k,0,0)$. We also note that for $k>0$, $\phi_{k,>(k-1,0,0)}(z)=\phi_{-k,<(-k+1,0,0)}(z)=1$. Thus, the assumptions (A1)--(A5) clearly hold.

Thus the proposition follows from Theorem \ref{main thm}.
\end{proof}

One can think about the reference state as the union of all negative boxes with 
coordinates
$(j,0,0)$, $j\in\Z_{<0}$. 
Then for all $k\in\Z$
the state $k$ is the set consisting of
all boxes $(j,0,0)$, $j\in\Z_{<k}$. Thus adding a negative box to a state is actually 
removing a box from this infinite chain of boxes.

The vector representation $U_1$ of color $1$ is obtained by interchanging  colors. The module $U_1$ is isomorphic to $U_0$ shifted by $q_1$.  The representation $U_0$ shifted by $q_1^2$ is isomorphic to $U_0$.

Another vector representation could be obtained by exchanging $q_1$ and $q_3$.

More vector representations can be obtained by ``shifts" --
by twisting with the shift of spectral parameter automorphism $\tau_a$.

\medskip

Clearly, the characters of the vector representations are given by
\be
\chi_{U_i}(z_0,z_1)=(1+z_i)\sum_{k\in\Z} (z_0z_1)^k=(1+z_i)\,\delta(z_0z_1).
\ee 

The module $U_i$ is an irreducible module of vertical $U_q\widehat{\gl}_2$. It is a quantization of the natural $\gl_2[t,t^{-1}]$-module $\C^2[t,t^{-1}]$.

\subsection{Fock modules}
Next we define the Fock modules $\mathcal  F_i$ as combinatorial $\E$ modules.

Let $\Psi(\mathcal F_0)=0_{q^{-1}}$.

Let the set of boxes consist of boxes with coordinates $(i,0,j)$, and color $i+j$ mod 2, where $i,j\in\Z_{\geq 0}$. All the boxes are positive, we have no negative boxes.  The ordering of boxes is lexicographic.

A partition $\la$ is a non-increasing sequence of non-negative numbers $\la_1\geq \la_2\geq\dots\geq \la_s\geq 0$ called parts. We do not distinguish between partitions which differ by zero parts.  We identify a partition with a set of boxes using the Young diagram. A partition $\la$ corresponds to the set of boxes with coordinates $(i,0,j)$, $j=0,\dots,s-1$, $i=0,\dots,\la_{j-1}-1$.

Let the set of states $S(\mathcal F_0)$ consist of all partitions. 
Let the empty partition be the reference state.

An example of a state is pictured on Figure \ref{Fock state}.
(Recall that $q_1,q_2,q_3$ correspond to coordinates $x,y,z$ respectively and the position of the box is $p(\Box)=q_1^{-x(\Box)} q_2^{-y(\Box)}q_3^{-z(\Box)}$.)

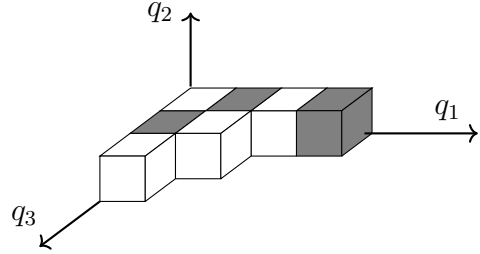
\begin{figure}[ht]
\centering
\begin{tikzpicture}
\bbox at (3,0);
\fbbox at (2.6,-0.9);
\rbbox at (3.2,-0.9);

\fwbox at (2.0,-0.9);
\wbox at (2.4,0);

\bbox at (1.8,0);

\bbox at (0.8, -0.3);
\wbox at (0.4, -0.6);
\fwbox at (0, -1.5);
\rwbox at (0.6, -1.5);
\wbox at (1.4, -0.3);
\fwbox at (1.0, -1.2);
\rwbox at (1.6, -1.2);
\wbox at (1.2, 0.0);


\draw[thick,->] (3.5,-0.6)--(5,-0.6);
\node at (4.6,-0.3)  {$q_1$};
\draw[thick,->] (1.2,0.02)--(1.2,1.0);
\node at (0.8,1.0)  {$q_2$};
\draw[thick,->] (0,-1.5)--(-0.8,-2.1);
\node at (-1,-1.7)  {$q_3$};

\end{tikzpicture}
\caption{A state in the Fock module.}
\label{Fock state}
\end{figure}

This state corresponds to the partition $\la=(4,2,1)$. 
We color the boxes by white and black colors, corresponding to colors $0$ and $1$. Thus colors alternate in both $x$- and $z$-directions. 

There are 4 concave boxes: $(4,0,0)$ (white), $(2,0,1),(1,0,2),(0,0,3)$ 
(black) and 
3 convex boxes: $(3,0,0)$ (black), $(1,0,1),(0,0,2)$ 
(white). Accordingly, the $\ell$-weight of this state is
\be
\phi_{(4,2,1)}=0_{q^{-1} q_1^{-4}}0^{-1}_{qq_1^{-1}q_3^{-1}}0^{-1}_{qq_3^{-2}} 1^{-1}_{qq_1^{-3}}1_{q^{-1}q_1^{-2}q_3^{-1}}1_{q^{-1}q_1^{-1}q_3^{-2}}1_{q^{-1}q_3^{-3}}.
\ee
Informally, one can think that this state is a tensor product of three vectors corresponding to rows: state $4$ from $U_0$, state $2$ from $U_1$ shifted by $q_3^{-1}$, and state $1$ from $U_0$ shifted by $q_3^{-2}$.

\medskip

\begin{prop}\label{Fock prop}
The space $\mathcal F_0=V(S(\mathcal F_0),\Psi(\mathcal F_0))$ with action \eqref{action} is an irreducible tame highest $\ell$-weight module of highest $\ell$-weight $\Psi(\mathcal F_0)$. In particular it has level $q$. 
\end{prop} 
\begin{proof}
The poles of the $\ell$-weight of a partition correspond to the convex and concave boxes. Indeed, the $\ell$-weight of row $i$ is a shift of 
$\phi_{\la_i}
\Psi(U_0)^{-1}$
by $q_3^{-i}$ if $i$ is even and similarly with the change of colors if $i$ is odd, see \eqref{vector monomial}. Then 
$\Psi(\mathcal F_0)$ multiplied by shifts of 
$\Psi(U_0)^{-1}$ 
and  
$\Psi(U_1)^{-1}$ 
produces only a pole in the convex box in the first column of $\la$. 

Clearly, a partition cannot have two convex boxes with the same position. A partition cannot have two concave boxes with the same position either. Therefore all $\ell$-weights are distinct. Thus the assumptions (A1)--(A5) hold.

Thus the proposition follows from Theorem \ref{main thm}.
\end{proof}

The module $\mathcal F_0$  described in Proposition \ref{Fock prop} is called the Fock module of color $0$. The Fock module $\mathcal F_0$  is the unique irreducible highest $\ell$-weight $\E$-module with highest $\ell$-weight $\Psi(\mathcal F_0)=0_{q^{-1}}=(\frac{qz-q^{-1}}{z-1},1)$. 

The Fock module $\mathcal F_1$ of color $1$ is obtained from $\mathcal F_0$ by changing colors.
The Fock module $\mathcal F_{1}$  is the unique irreducible highest $\ell$-weight $\E$-module with highest $\ell$-weight $\Psi(\mathcal F_1)=1_{q^{-1}}$. 

The $\E$ Fock modules in parallel realization were constructed in  \cite{FJMM1} using the inductive limit of tensor products of vector representations, an analog of the semi-infinite wedge construction.

\medskip 

The character of the Fock space does not factorize, but we have the usual partition formula when $z_0=z_1=z$.
\be
\chi_{\mathcal F_i}(z,z)=\frac{1}{\prod_{i=1}^\infty (1-z^i)}.
\ee 

The Fock module $\mathcal F_0$ restricted to the vertical  $U_q^v\widehat{\gl}_2\subset \E$ subalgebra is the irreducible basic module of level $q$. It is generated by a singular vector $v$ such that $K_0v=qv$ and $K_1v=v$.
If $q=q_3$,
then $\mathcal F_0$ is also an evaluation module and $\Psi(\mathcal F_0)$ has the form 
\eqref{evaluation weight} with $a=q$, $b=1$, and $u=q^{-1}$. 

\begin{rem} \rm{As an exercise, remark that one can describe the Fock module in a slightly different form by declaring box $(1,0,1)$ negative. Then the empty configuration would not be a highest weight vector but the vector corresponding to one box in the standard realization described above. In particular, we would have $\Psi=0_q^{-1}1_{q_1q}1_{q_3q}$. \qed}
\end{rem}

\subsection{Macmahon modules.} 
We now define the Macmahon module $M_0^\kappa$ with $\kappa\in\C^{\times}$.

Let $\kappa$ be generic. 
It means $\kappa\not\in q_1^\Z q_2^\Z$. 
Let $$\Psi(M_0^\kappa)=\Big(\frac{z \kappa-\kappa^{-1}}{z-1},1\Big).$$

Let the set of boxes consist of boxes with coordinates $(i,j,k)$, and color $i+k$ mod 2, where $i,j,k\in\Z_{\geq 0}$. All the boxes are positive, we have no negative boxes. The ordering of boxes is lexicographic.

For partitions $\la,\mu$ we say $\mu\subset \la$ if for all $j$, $\la_j\geq \mu_j$. 

A plane partition $\bs \la$ is a
set of partitions $\la^{(1)}\supset \la^{(2)}\supset \dots\supset \la^{(s)}$ called layers. We do not distinguish between plane partitions which differ by zero layers.

The set of states $S(M_0^\kappa)$ is the set of all plane partitions.  The reference state is the empty plane partition.

The Young diagram of a plane partition $\bs \la$ is a union of Young diagrams of layers $\la^{(j)}$ placed on height $j-1$. The boxes corresponding to $\la^{(j)}_i$ have coordinates $(0,j-1,i-1),(1,j-1,i-1),\dots, (\la^{(j)}_i-1,j-1,i-1)$. 

Note that the colors alternate along $q_1$ and $q_3$ axes and repeat along $q_2$ axes. The starting box with coordinates $(0,0,0)$ has color $0$.

An example of a state in $M_0^\kappa$ is given in Figure \ref{Macmahon picture}. In the figure, $\la^{(1)}=(4,2,1),$ $\la^{(2)}=(3,1)$, and $\la^{(3)}=(1)$.

\begin{figure}[ht]
\centering
\begin{tikzpicture}
\bbox at (3,0);
\fbbox at (2.6,-0.9);
\rbbox at (3.2,-0.9);
\rwbox at (2.6,-0.3);
\fwbox at (2.0,-0.9);
\fwbox at (2.0,-0.3);
\wbox at (2.4,0.6);
\bbox at (1.8,0.6);
\rwbox at (1.4,0.3);
\rbbox at (1.0,-0.6);
\wbox at (1.2, 1.2);
\fwbox at (0.8, 0.3);
\bbox at (0.8, 0.3);
\fbbox at (0.4, -0.6);
\wbox at (0.4, -0.6);
\fwbox at (0, -1.5);
\rwbox at (0.6, -1.5);
\wbox at (1.4, -0.3);
\fwbox at (1.0, -1.2);
\rwbox at (1.6, -1.2);
\fbbox at (1.4, -0.3);
\draw[thick,->] (3.5,-0.6)--(5,-0.6);
\node at (4.6,-0.3)  {$q_1$};
\draw[thick,->] (1.2,1.2)--(1.2,2.2);
\node at (1.6,2.0)  {$q_2$};
\draw[thick,->] (0,-1.5)--(-0.8,-2.1);
\node at (-1,-1.7)  {$q_3$};

\end{tikzpicture}
\caption{A state in the Macmahon module.}
\label{Macmahon picture}
\end{figure}
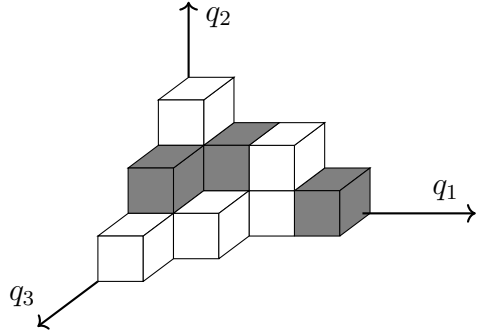

\begin{prop}\label{Macmahon prop}
The space $M_0^\kappa=V(S(M_0^\kappa),\Psi(M_0^\kappa))$ with action \eqref{action} is an irreducible tame highest $\ell$-weight
$\E$-module of highest $\ell$-weight $\Psi(M_0^\kappa)$. In particular, the level of  $M_0^\kappa$ is $\kappa$.
\end{prop} 
\begin{proof}
The proof is similar to the proof of Proposition \ref{Fock prop}.
\end{proof}
We call the module  $M_0^\kappa$ constructed in Proposition \ref{Macmahon prop} the Macmahon module of color $0$.

The Macmahon module $M_1^\kappa$ of color $1$ is the irreducible highest $\ell$-weight $\E$-module with highest $\ell$-weight $\Psi(M_1^\kappa)=(1,\frac{z\kappa-\kappa^{-1}}{z-1})$. It is obtained from $M_0^\kappa$ by exchanging colors.

The Macmahon modules are building blocks for highest $\ell$-weight $\E$-modules. The  Macmahon module was constructed explicitly in  \cite{FJMM1} using inductive limit construction of tensor products of Fock modules.

\medskip 

We note the combinatorial difference between the Fock module $F_0$ and Macmahon module $M_0^\kappa$. Both modules are generated from the empty reference state which has a pole $z=1$ in the zero component of the $\ell$-weight. After adding a white box the $\ell$-weight is multiplied by $A^{-1}_{0,1}$. This creates two poles at $z=q_1^{-1}$ and $z=q_3^{-1}$ in the first component of the $\ell$-weight. The difference is in the zero component. In the case of the Fock module, the zero $z=q^{-2}$ of $\Psi(F_0)=0_{q^{-1}}$ cancels the pole of  $A^{-1}_{0}$ and no more poles are created. While in the Macmahon case such a cancellation does not happen and therefore, one obtains an additional concave box with coordinates $(0,1,0)$. In addition, since the zero of $\Psi(M_0^\kappa)$ is not used, it can be arbitrary and therefore the Macmahon module depends on a parameter $\kappa$.

\medskip

For special values of 
$\kappa$ the Macmahon module $M_0^\kappa$ may develop a submodule due to the factor $(\kappa p(\Box)-\kappa^{-1})$ in the coefficients $b_{\bs \la,  \Box}$ in the action of $E_0(z)$. 
To take advantage of this construction, choose a white box $\Box$, $c(\Box)=0$ with coordinates $(i,j,k)$ such that $\min\{i,j,k\}<1$ and specialize 
$\kappa^2=p(\Box)^{-1}=q_1^iq_2^jq_3^k$. 
Then the action of $E_0(z)$ cannot remove the chosen box due to vanishing of the coefficient $b_{\bs \la,  \Box}$. As the result we obtain an irreducible tame $\E$-module $\bar M_0^\kappa$ whose basis is given by all plane partitions which do not contain the chosen white box, see \cite{FJMM1}.

The following two cases are of 
special importance since their restrictions to the horizontal $U_q\widehat{\gl_2}$ are irreducible.

Let  $\kappa=q$. Then the prohibited box has coordinates $(0,1,0)$ and 
the corresponding module has a basis consisting
of colored plane partitions with one layer. Moreover, $\Psi(M_0^q)=\Psi(F_0)$. So the irreducible quotient $\bar M_0^q$ is the Fock module $F_0$.


We show a state in the Fock space $F_0$ corresponding to partition $\la=(4,2,1)$ in Figure \ref{Fock picture}. The prohibited box is shown dashed.

\begin{figure}[ht]
\centering
\begin{tikzpicture}
\bbox at (3,0);
\fbbox at (2.6,-0.9);
\rbbox at (3.2,-0.9);

\fwbox at (2.0,-0.9);
\wbox at (2.4,0);

\bbox at (1.8,0);

\bbox at (0.8, -0.3);
\wbox at (0.4, -0.6);
\fwbox at (0, -1.5);
\rwbox at (0.6, -1.5);
\wbox at (1.4, -0.3);
\fwbox at (1.0, -1.2);
\rwbox at (1.6, -1.2);
\wbox at (1.2, 0.0);

\draw[dashed] (0.8,0.3)--++(0.4,0.3)--++(0.6,0.0);
\draw[dashed] (1.2,0.6)--++(0,-0.6);
\fdbox  at (0.8,-0.3);
\rdbox  at (1.4,-0.3);

\draw[thick,->] (3.5,-0.6)--(5,-0.6);
\node at (4.6,-0.3)  {$q_1$};
\draw[thick,->] (1.2,0.62)--(1.2,1.6);
\node at (1,2.0)  {$q_2$};
\draw[thick,->] (0,-1.5)--(-0.8,-2.1);
\node at (-1,-1.7)  {$q_3$};

\end{tikzpicture}
\caption{The prohibited box and the Fock module.}
\label{Fock picture}
\end{figure}
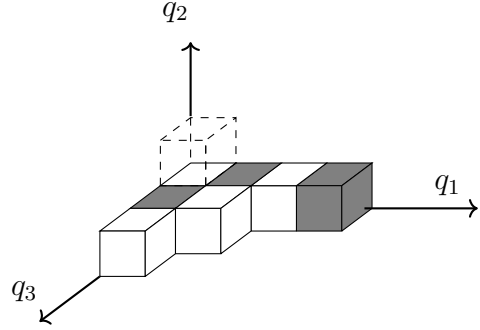

\medskip

Let $\kappa=q_3$. Then the prohibited box has coordinates $(0,0,2)$. We denote the corresponding irreducible $\E$-module $\bar M_0^{q_3}$ by $G_0$. One can say that this module has a basis labeled by plane partitions with two layers placed vertically. Importantly, the coloring is different since vertically the coloring does not change. 


 We show a state in the module $G_0$ corresponding to the
partition $\bs \la=((4,3),(2,1),(1),(1))$ in Figure \ref{G picture}. The same picture can be seen as two layer partition $\bs \mu=((4,2,1,1),(2,1,1))$ placed vertically. The prohibited box is shown dashed.

Note that we have a similar module corresponding to the choice $\kappa=q_1$ (the prohibited box has coordinates $(2,0,0)$). 
Also, if one starts from $M_1$, the same choices of $\kappa$ produces two similar modules. The difference in pictures is the color of the initial box with coordinates $(0,0,0)$. 

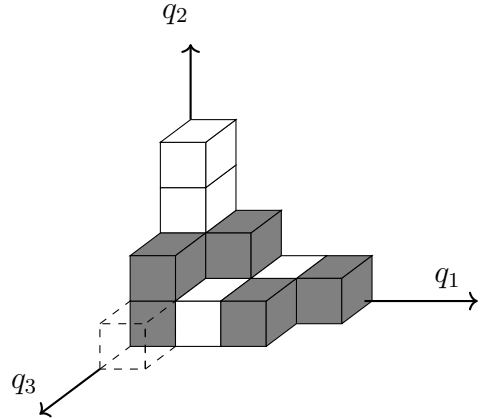
\begin{figure}[ht]
\centering
\begin{tikzpicture}
\bbox at (3,0);
\fbbox at (2.6,-0.9);
\rbbox at (3.2,-0.9);



\bbox at (1.8,0.6);
\rbbox at (1.0,-0.6);
\rwbox at (1.4,0.3);
\fwbox at (0.8, 0.3);
\rwbox at (1.4,0.9);
\fwbox at (0.8, 0.9);
\wbox at (1.2, 1.8);
\bbox at (0.8, 0.3);
\fbbox at (0.4, -1.2);
\rbbox at (1.0, -0.6);
\fbbox at (0.4, -0.6);

\fwbox at (1.0, -1.2);
\wbox at (1.4, -0.3);
\fbbox at (1.6, -1.2);
\rbbox at (2.2, -1.2);
\bbox at (2.0, -0.3);
\fbbox at (1.4, -0.3);
\rbbox at (2.0, -0.3);

\wbox at (2.4, 0);


\fdbox at (0, -1.5);
\rdbox at (0.6, -1.5);
\draw[dashed] (0,-1.5)--++(0.4,0.3);
\draw[dashed] (0,-0.9)--++(0.4,0.3);

\draw[thick,->] (3.5,-0.6)--(5,-0.6);
\node at (4.6,-0.3)  {$q_1$};
\draw[thick,->] (1.2,1.8)--(1.2,2.8);
\node at (1,3.2)  {$q_2$};
\draw[thick,->] (0,-1.5)--(-0.8,-2.1);
\node at (-1,-1.7)  {$q_3$};

\end{tikzpicture}
\caption{A state in the $G_0$ module.}
\label{G picture}
\end{figure}

\medskip

The character of the Macmahon modules is given by, see Theorem 7.6 in \cite{BM},
\be
\chi_{M_0^\kappa}(z_0,z_1)=\frac{1}{\prod_{i=1}^\infty (1-(z_0z_1)^{i})^{2i}(1-z_0(z_0z_1)^{i-1})^{i}(1-z_1(z_0z_1)^{i-1})^{i-1}}.
\ee 
In particular, we have the famous Macmahon formula for plane partitions, 
\be
\chi_{M_0^\kappa}(z,z)=\frac{1}{\prod_{i=1}^\infty (1-z^i)^i}.
\ee

\medskip 

The Macmahon module $M_0^\kappa$ restricted to the vertical  $U_q^v\widehat{\gl}_2\subset \E$ subalgebra is not irreducible. The reference state $\ket{\emptyset}$ is a highest $\ell$-weight vector of weight $(\kappa,1)$ and therefore the irreducible $U_q^v\widehat{\gl}_2$-subquotient containing the reference state is a parabolic Verma module. When $\kappa=q_3$, the module $G_0$ is the evaluation parabolic Verma module. In particular, the weight $\Psi(M_0^{q_3})$ has form \eqref{evaluation weight} with $a=q_3$, $b=1$, and $u=q_3^{-1}$. Note that $q_3$ is a generic number with respect to the parameter  $q$ of $U_q^v\widehat{\gl}_2$.

\subsection{Evaluation Verma modules}
Let $\mu$ be a non-integer complex number.
Let 
\be
\Psi(G_0^\mu)=\Big(q_3q^{-\mu}\frac{z-q_3^{-2}}
{z- q_2^{-\mu}},\ q^\mu\frac{z-q_3^{-1}q_2^{-\mu}}{z-q_3^{-1}}\Big).
\ee

Let the set of boxes consist of boxes with  positions $(i,k,1)$, color $i+1$ mod 2, and of boxes with positions $(i,k+\mu,0)$, color $i$ mod 2.  Here $i,k\in\Z_{\geq 0}$. All the boxes are positive, we have no negative boxes.
The ordering of boxes with the same $z$-coordinate is lexicographic. The boxes with $z$-coordinate $1$ are larger than the boxes with $z$-coordinate $0$. 

In fact, the last condition actually could be opposite.

The set of states $S(G_0^\mu)$ is the set of pairs of partitions positioned vertically.  The reference state is the pair of empty partitions.

A state in  $S(G_0^\mu)$ is shown on Figure \ref{verma picture}. Here we read the two partitions vertically: $(4,2,1,1)$ and $(2,1,1)$. Since the position of the first partition is shifted by $(0,\mu,0)$ we picture it on a pedestal of height $\mu$ which is infinite in direction of $q_1$ and one box thick (shaded).

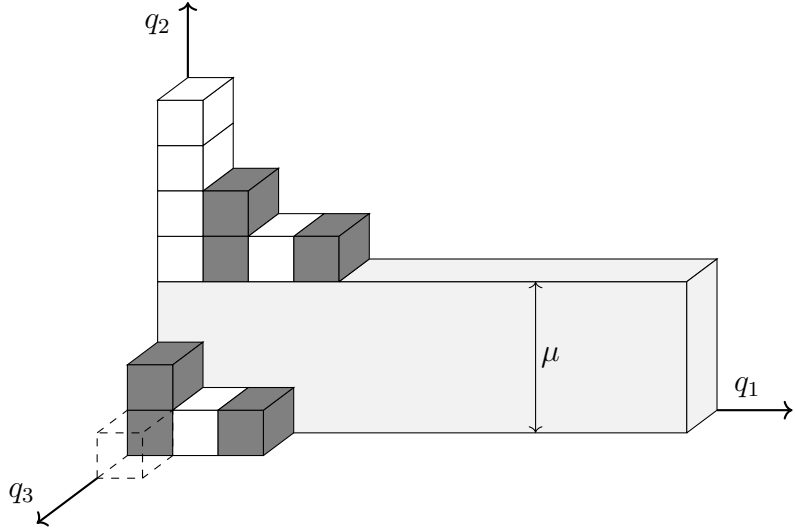
\begin{figure}[ht]
\centering
\begin{tikzpicture}
\filldraw[fill=gray!10] (0.8,-0.9)--++(0,2)--++(7,0)--++(0,-2)--(0.8,-0.9);
\filldraw[fill=gray!10] (3.2,1.1)--++(0.4,0.3)--++(4.6,0)--++(-0.4,-0.3)--(3.2,1.1);
\filldraw[fill=gray!10] (7.8,-0.9)--++(0,2)--++(0.4,0.3)--++(0,-2)--(7.8,-0.9);

\bbox at (3,2);
\fbbox at (2.6,1.1);
\rbbox at (3.2,1.1);



\bbox at (1.8,2.6);
\rwbox at (1.4,2.3);
\fwbox at (0.8, 2.3);
\rwbox at (1.4,2.9);
\fwbox at (0.8, 2.9);
\wbox at (1.2, 3.8);
\bbox at (0.8, 0.3);
\fbbox at (0.4, -1.2);
\rbbox at (1.0, -0.6);
\fbbox at (0.4, -0.6);

\fwbox at (1.0, -1.2);
\wbox at (1.4, -0.3);
\fbbox at (1.6, -1.2);
\rbbox at (2.2, -1.2);
\bbox at (2.0, -0.3);
\fbbox at (1.4, 1.7);
\rbbox at (2.0, 1.7);
\fbbox at (1.4, 1.1);
\fwbox at (2.0, 1.1);
\wbox at (2.4,2.0);
\fwbox at (0.8,1.1);
\fwbox at (0.8,1.7);


\fdbox at (0, -1.5);
\rdbox at (0.6, -1.5);
\draw[dashed] (0,-1.5)--++(0.4,0.3);
\draw[dashed] (0,-0.9)--++(0.4,0.3);

\draw[thick,->] (8.2,-0.6)--(9.2,-0.6);
\node at (8.6,-0.3)  {$q_1$};
\draw[thick,->] (1.2,3.8)--(1.2,4.8);
\node at (0.8,4.5)  {$q_2$};
\draw[thick,->] (0,-1.5)--(-0.8,-2.1);
\node at (-1,-1.7)  {$q_3$};
\draw[<->] (5.8,-0.9)--(5.8,1.1);
\node at (6,0.1)  {$\mu$};

\end{tikzpicture}
\caption{A state in the Verma module $G_0^\mu$.}
\label{verma picture}
\end{figure}

\begin{prop}\label{Verma prop}
The space $G_0^\mu=V(S(G_0^\mu),\Psi(G_0^\mu))$ with action \eqref{action} is an irreducible tame highest $\ell$-weight $\E$-module of highest $\ell$-weight $\Psi(G_0^\mu)$. In particular, the level of $G_0^\mu$ is $q_3$.
\end{prop} 
\begin{proof}
The proof is similar the proof of Proposition \ref{Fock prop}.
\end{proof}
Alternatively, one can construct the module $G_0^\mu$ by starting with the $G_0$ module. Consider states in $G_0$ which contain all the boxes in the pedestal (with positions $(i,j,0)$ where $i=0,1,\dots,\mu-1,$ $j=0,1,\dots,\la-1$ with
(large) positive integers $\la,\mu$) and observe
that the matrix coefficients stabilize as $\la\to \infty$ and can be continued with respect to $\mu$ to generic values. We do not give further details.

\medskip

The character of the module $G_0^\mu$ is a product 
of those
of vertical partitions. Unlike the case of the Fock modules, the character of the vertical partitions  is readily computed similarly to the non-graded case. One simply reads the partition along the columns.
Set $t=z_0z_1$. Thus, the character of one vertical partition whose top box has color $i$ is 
\be
\bar\chi_i(z_0,z_1)=\frac{1}{\prod_{j=1}^\infty (1-t^{j}) (1-z_it^{j-1})},
\ee
and
\be
\chi_{G_0^\mu}(z_0,z_1)=\bar\chi_0(z_0,z_1)\bar\chi_1(z_0,z_1)=\frac{1}{\prod_{i=1}^\infty (1-t^{i})^2 (1-z_0t^{i-1})(1-z_1t^{i-1})}.
\ee

The module $G_0^\mu$ restricted to the vertical  $U_q^v\widehat{\gl}_2\subset \E$ subalgebra is the Verma module
of weight $(q_3 q^{-\mu},q^\mu)$. Moreover, $G_0^\mu$ is the evaluation module obtained from the Verma module, and $\Psi(G_0^\mu)$ is given by \eqref{evaluation weight} with $a=q_3q^{-\mu},b=q^\mu$, and $u=q^{-\mu} q_3^{-1}$.

The module $G_1^\mu$ is obtained from $G_0^\mu$ by interchanging colors.
We have the equality $\Psi(G_1^\mu)(z)=\Psi(G_0^{\beta-\mu})(q^{-\beta+2\mu}z)$, where $\beta$ is defined by $q_3=q^\beta$. Thus,
up to a twist by the shift automorphism $\hat\tau_a$,
the module $G_1^\mu$ is isomorphic to $G_0^{\beta-\mu}$.

\subsection{Evaluation relaxed Verma modules}
Let $\mu,\nu$ be complex numbers. We assume that $\mu,\nu$ and $\mu+\nu$ are not integers.
Let 
\be
\Psi(G_0^{\mu,\nu})=
\Big(q_3q^{-\mu-2\nu} \frac{(z-q_3^{-2})(z-q_2^{-\mu+1})}
{(z- q_2^{-\mu-\nu})(z- q_2^{-\mu-\nu+1})},\ q^{\mu+2\nu}\frac{(z-q_3^{-1}q_2^{-\mu-\nu})(z-q_1^{-1}q_2^{-\mu-\nu})}{(z-q_3^{-1})(z-q_1^{-1}q_2^{-\mu})}\Big).
\ee

Let the set of positive boxes consist of boxes with coordinates $(i,k,1)$, color $i+1$ mod 2,  of boxes with coordinates $(i+1,k+\mu,0)$, color $i+1$ mod 2 and of boxes with positions $(0,\mu+\nu+k,0)$, color $0$. Here $i,k\in\Z_{\geq 0}$. In addition, we have negative boxes with coordinates $(0,\mu+\nu-k,0)$, color $0$, $k\in\Z_{>0}$. 

Thus we have three families of boxes. Each family is ordered lexicographically. In addition, we declare that boxes from the
last group are larger than boxes in the second group which are in their turn larger than boxes in the first group. (The last assignment can be changed arbitrarily.)

The set of states $S(G_0^{\mu,\nu})$ is the set of pairs of partitions positioned vertically and an integer $k$. The integer $k$ indicates the number of boxes put on the box with coordinates $(0,\mu+\nu,0)$ 
(negative $k$ means removing $|k|$ boxes). 
The reference state is the pair of empty partitions and integer $0$.

A state in  $S(G_0^{\mu,\nu})$ is shown on Figure \ref{relaxed verma picture}. Here we read the two partitions vertically: $(2,2,1)$ and $(2,1)$ and $k=2$.

Since the position of the boxes above $(0,0,0)$ are shifted by $(0,\mu+\nu,0)$ we picture it on a tower of height $\nu$ which is placed on the previous pedestal (shaded). Note that for negative $k$, $|k|$ boxes are taken off that tower.

Alternatively, one can start with the $G_0^\mu$ module. Then consider states which contain the boxes in the tower above $(0,0,0)$ (with positions $(0,j+\mu,0)$ where $j=0,1,\dots,\nu-1$ with (large) positive integers $\nu$) and observes that the matrix coefficients can be continued with respect to $\nu$ to generic values. We do not give further details.

For that reason we picture the tower over the first box on the pedestal. Though, since $\nu$ is generic, we could place it over any prohibited white box.
 
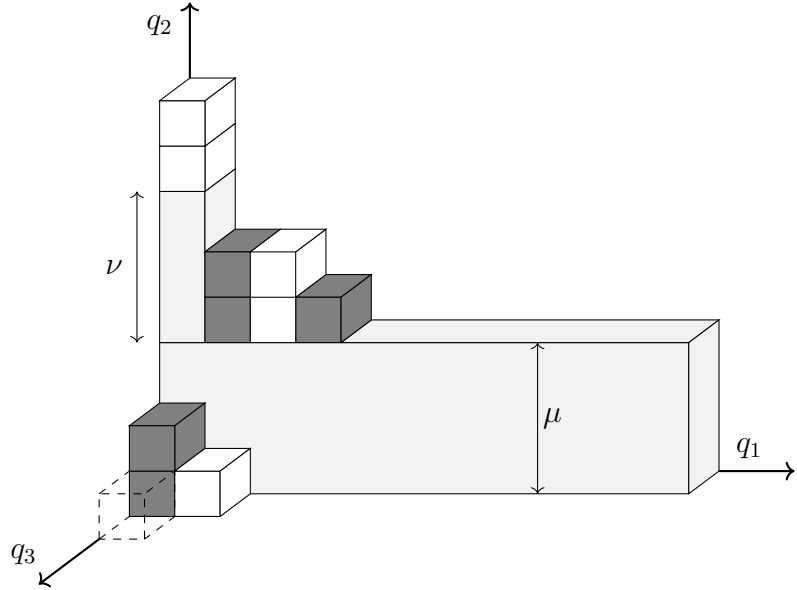
\begin{figure}[ht]
\centering
\begin{tikzpicture}
\filldraw[fill=gray!10] (0.8,-0.9)--++(0,2)--++(7,0)--++(0,-2)--(0.8,-0.9);
\filldraw[fill=gray!10] (3.2,1.1)--++(0.4,0.3)--++(4.6,0)--++(-0.4,-0.3)--(3.2,1.1);
\filldraw[fill=gray!10] (7.8,-0.9)--++(0,2)--++(0.4,0.3)--++(0,-2)--(7.8,-0.9);
\filldraw[fill=gray!10] (0.8,1.1)--++(0,2)--++(0.6,0)--++(0,-2)--(0.8,1.1);
\filldraw[fill=gray!10] (1.4,3.1)--++(0.4,0.3)--++(0,-0.8)--++(-0.4,-0.3)--(1.4,3.1);

\bbox at (3,2);
\fbbox at (2.6,1.1);
\rbbox at (3.2,1.1);



\bbox at (1.8,2.6);
\rwbox at (1.4,3.1);
\fwbox at (0.8, 3.1);
\rwbox at (1.4,3.7);
\fwbox at (0.8, 3.7);
\wbox at (1.2, 4.6);
\bbox at (0.8, 0.3);
\fbbox at (0.4, -1.2);
\rbbox at (1.0, -0.6);
\fbbox at (0.4, -0.6);

\fwbox at (1.0, -1.2);
\wbox at (1.4, -0.3);
\rwbox at (1.6, -1.2);
\fbbox at (1.4, 1.7);
\rwbox at (2.6, 1.7);
\fbbox at (1.4, 1.1);
\fwbox at (2.0, 1.1);
\wbox at (2.4,2.6);


\fdbox at (0, -1.5);
\rdbox at (0.6, -1.5);
\draw[dashed] (0,-1.5)--++(0.4,0.3);
\draw[dashed] (0,-0.9)--++(0.4,0.3);

\draw[thick,->] (8.2,-0.6)--(9.2,-0.6);
\node at (8.6,-0.3)  {$q_1$};
\draw[thick,->] (1.2,4.6)--(1.2,5.6);
\node at (0.8,5.3)  {$q_2$};
\draw[thick,->] (0,-1.5)--(-0.8,-2.1);
\node at (-1,-1.7)  {$q_3$};
\draw[<->] (5.8,-0.9)--(5.8,1.1);
\node at (6,0.1)  {$\mu$};
\draw[<->] (0.5,1.1)--(0.5,3.1);
\node at (0.2,2.1)  {$\nu$};

\end{tikzpicture}
\caption{A state in the evaluation relaxed Verma module $G_0^{\mu,\nu}$.}
\label{relaxed verma picture}
\end{figure}

\begin{prop}\label{relaxed Verma prop}
The space $G_0^{\mu,\nu}=V(S(G_0^{\mu,\nu}),\Psi(G_0^{\mu,\nu}))$ with action \eqref{action} is an irreducible tame  $\E$-module of level $q_3$. 
\end{prop} 
\begin{proof}
The proof follows the proof of Proposition \ref{Fock prop}.
\end{proof}


\medskip 

The character of $G_0^{\mu,\nu}$ is easy to compute. We set $t=z_0z_1$.
\begin{align*}
\chi_{G_0^{\mu,\nu}}(z_0,z_1)&= \bar\chi_1^2(z_0,z_1)\,\sum_{i\in \Z} z_0^i=\frac{1}{\prod_{j=1}^{\infty}(1-t^j)^2(1-z_1t^{j-1})^2}\,\sum_{i\in \Z} z_0^i\\
&=\frac{1}{\prod_{j=1}^{\infty}(1-t^j)^2(1-z_0^{-1}t^{j})^2}\,\sum_{i\in \Z} z_0^i= \frac{1}{\prod_{j=1}^\infty(1-t^j)^4} \, \sum_{i\in \Z} z_0^i= \frac{1}{\prod_{j=1}^\infty(1-t^j)^4}\, \delta(z_0).
\end{align*}
Here we used the identity $\delta(z)f(z)=f(1)\delta(z)$ which holds for any power series $f(z)$ in $z^{-1}$.

\medskip 
The character of $G_{1}^{\mu,\nu}$ coincides with the character of 
$\hat L(a,b,c;\kappa)$, cf. Figure \ref{flat pic}. Note that by \eqref{grading}, the degree of a black box on Figure \ref{flat pic} is $(-1,0)$ while the degree of a  white box is $(1,-1)$.

The module $G_1^{\mu,\nu}$ restricted to the vertical  
$U_q^v\widehat{\gl}_2\subset \E$ 
subalgebra is the irreducible relaxed Verma module  $\hat L(a,b,c;\kappa)$ with 
$q^a=q_3q^{-\mu-2\nu}$,  
$q^b=1$, $c=(q^\nu-q^{-\nu})(q^{a+1+\nu}-q^{-a-1-\nu})/(q-q^{-1})^2$
and  $\kappa=q_3$, so that $\mathcal{C}=q_3q^{-\mu+1}+q_3^{-1}q^{\mu-1}$. Moreover, $G_1^{\mu,\nu}$ is the evaluation module obtained from $\hat L(a,b,c;\kappa)$.

Similarly, the module $G_0^{\mu,\nu}$ is the evaluation module obtained from slope one relaxed module $\hat L_1(a,b,c;\kappa)$.

\subsection{Evaluation slanted relaxed Verma modules}
Let $\mu,\nu$ be complex numbers. We assume that $\mu,\nu$ and $\mu+\nu$ are not integers. Let $m\geq 1$.

Let 
\begin{align*}
&\Psi(G_{0,m}^{\mu,\nu})=(\Psi^0(G_{0,m}^{\mu,\nu}),\Psi^1(G_{0,m}^{\mu,\nu})), \\
&\Psi^0(G_{0,m}^{\mu,\nu})=q^{-\mu-2\nu}q_3^{-m+1}
\frac{ (z-q_3^{-2})(z-q_1^{-1}q_3^{-1}q_2^{-\mu})\prod_{i=1}^{m-1}(z-q_2^{-\mu-\nu}q_1^iq_3^{-i})}{\prod_{i=0}^m(z-q_2^{-\mu-\nu+1}q_1^iq_3^{-i})},
\\
&\Psi^1(G_{0,m}^{\mu,\nu}) =q^{\mu+2\nu}q_3^m\frac{\prod_{i=0}^{m+1}(z-q_1^{i-1}q_3^{-i}q_2^{-\mu-\nu})}{(z-q_3^{-1})(z-q_1^{-1}q_2^{-\mu})\prod_{i=0}^{m-1}(z-q_1^{i+1}q_3^{-i}q_2^{-\mu-\nu})}.
\end{align*}

It is important to note 
\be
\Psi(G_{0,m}^{\mu,\nu})  \prod_{i=0}^m A^{-1}_{0,q_1^{i}q_3^{-i}q_2^{-\mu-\nu}}\prod_{i=0}^{m-1}A^{-1}_{1,q_1^{i+1}q_3^{-i}q_2^{-\mu-\nu}}=\Psi(G_{0,m}^{\mu,\nu+1}).
\ee 
Here we take a product of $\ell$-weights of $m$ black boxes and $m+1$ white boxes forming a staircase as shown on Figure \ref{top view picture} (with $m=4$) with $y$-coordinate (height) $\mu+\nu$.

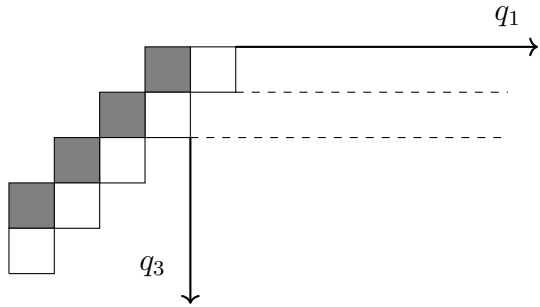
\begin{figure}[ht]
\centering
\begin{tikzpicture}
\fwbox at (1.8, 1.8);
\fbbox at (1.2, 1.8);
\fwbox at (1.2, 1.2);
\fbbox at (0.6, 1.2);
\fwbox at (0.6, 0.6);
\fbbox at (0.0, 0.6);
\fwbox at (0.0, 0.0);
\fbbox at (-0.6, -0.0);
\fwbox at (-0.6, -0.6);

\draw[thick,->] (2.4,2.4)--(6.4,2.4);
\node at (6,2.8)  {$q_1$};
\draw[thick,->] (1.8,1.2)--(1.8,-1);
\node at (1.3,-0.5)  {$q_3$};

\draw[dashed] (2.4,1.8)--(6.0,1.8);
\draw[dashed] (1.8,1.2)--(6.0,1.2); 

\end{tikzpicture}
\caption{A top view of the reference state in the slope $4$ slanted relaxed Verma module $G_{0,4}^{\mu,\nu}$.}
\label{top view picture}
\end{figure}

Let the set of positive boxes consist of boxes with coordinates $(i,k,1)$, color $i+1$ mod 2,  of boxes with coordinates $(i+1,k+\mu,0)$, color $i+1$ mod 2, 
of boxes with coordinates 
$(-s,\mu+\nu+k,s)$,
color $0$, with $s=0,\dots,m$, and  of boxes with coordinates  $(-1-s,\mu+\nu+k,s)$, color $1$, with $s=0,\dots,m-1$. Here $i,k\in\Z_{\geq 0}$. In addition, we have negative boxes with  boxes with coordinates  $(-s,\mu+\nu-k,s)$, color $0$, where $s=0,\dots,m$, and with coordinates $(-1-s,\mu+\nu-k,s)$, color $1$, with $s=0,\dots,m-1,$ $k\in\Z_{>0}$. 

In other words, we can have boxes vertically in the first two rows shown by dashed lines in Figure \ref{top view picture} and boxes above the staircase or below the staircase. The boxes below the staircase are negative. The boxes on the first row are lifted by $\mu$, the ones on the staircase by $\mu+\nu$.

Thus we have three families of boxes: on the bottom, on the pedestal and at the tower. Each family is ordered lexicographically. In addition, we declare that boxes from the
last group are larger than boxes in the second group which are in their turn larger than boxes in the first group. (The last assignment can be changed arbitrarily.)

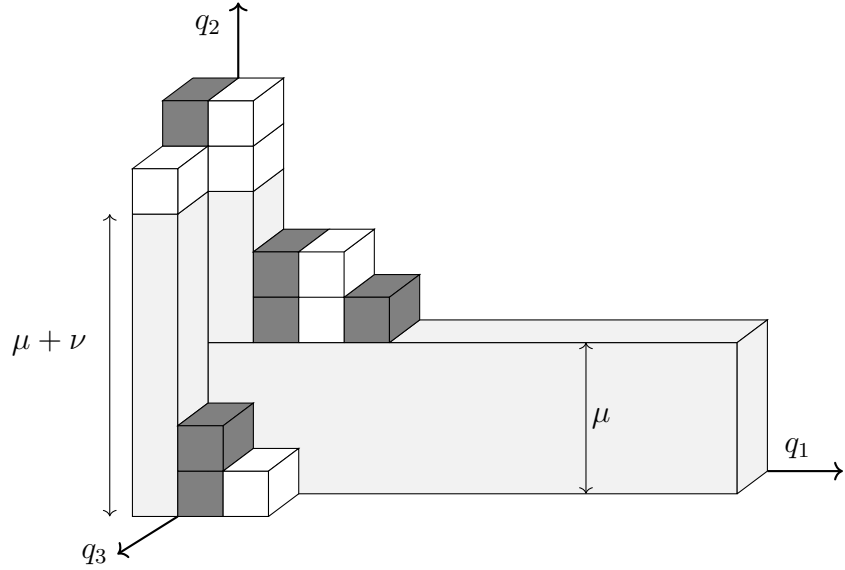
\begin{figure}[ht]
\centering
\begin{tikzpicture}
\filldraw[fill=gray!10] (0.8,-0.9)--++(0,2)--++(7,0)--++(0,-2)--(0.8,-0.9);
\filldraw[fill=gray!10] (3.2,1.1)--++(0.4,0.3)--++(4.6,0)--++(-0.4,-0.3)--(3.2,1.1);
\filldraw[fill=gray!10] (7.8,-0.9)--++(0,2)--++(0.4,0.3)--++(0,-2)--(7.8,-0.9);
\filldraw[fill=gray!10] (0.8,1.1)--++(0,2)--++(0.6,0)--++(0,-2)--(0.8,1.1);
\filldraw[fill=gray!10] (1.4,3.1)--++(0.4,0.3)--++(0,-0.8)--++(-0.4,-0.3)--(1.4,3.1);
\filldraw[fill=gray!10] (-0.2,2.8)--++(0.6,0.0)--++(0,-4)--++(-0.6,0.0)--(-0.2,2.8);
\filldraw[fill=gray!10] (0.4,2.8)--++(0.4,0.3)--++(0,-2.8)--++(-0.4,-0.3)--(0.4,2.8);

\bbox at (3,2);
\fbbox at (2.6,1.1);
\rbbox at (3.2,1.1);



\bbox at (1.8,2.6);
\rwbox at (1.4,3.1);
\fwbox at (0.8, 3.1);
\rwbox at (1.4,3.7);
\fwbox at (0.8, 3.7);
\wbox at (1.2, 4.6);
\bbox at (0.8, 0.3);
\fbbox at (0.4, -1.2);
\rbbox at (1.0, -0.6);
\fbbox at (0.4, -0.6);
\bbox at (0.6,4.6);
\wbox at (0.2,3.7);

\rwbox at (0.4,2.8);
\fwbox at (-0.2, 2.8);
\fbbox at (0.2, 3.7);

\fwbox at (1.0, -1.2);
\wbox at (1.4, -0.3);
\rwbox at (1.6, -1.2);
\fbbox at (1.4, 1.7);
\rwbox at (2.6, 1.7);
\fbbox at (1.4, 1.1);
\fwbox at (2.0, 1.1);
\wbox at (2.4,2.6);




\draw[thick,->] (8.2,-0.6)--(9.2,-0.6);
\node at (8.6,-0.3)  {$q_1$};
\draw[thick,->] (1.2,4.6)--(1.2,5.6);
\node at (0.8,5.3)  {$q_2$};
\draw[thick,->] (0.4,-1.2)--(-0.4,-1.7);
\node at (-0.7,-1.7)  {$q_3$};
\draw[<->] (5.8,-0.9)--(5.8,1.1);
\node at (6,0.1)  {$\mu$};
\draw[<->] (-0.5,-1.2)--(-0.5,2.8);
\node at (-1.3,1.1)  {$\mu+\nu$};

\end{tikzpicture}
\caption{A state in the slope $1$ slanted relaxed Verma module $G_{0,1}^{\mu,\nu}$.}
\label{slanted relaxed verma picture}
\end{figure}

The set of states $S(G_{0,m}^{\mu,\nu})$ is the set of pairs of partitions positioned vertically and integers $a_0,\dots,a_m$, $b_0,\dots,b_{m-1}$ such that 
\begin{align}\label{condition}
b_s\geq a_s,\quad  b_s\geq a_{s+1}, \qquad  s=0,\dots,m-1.
\end{align}
The integer $a_s$ indicates the number of the white boxes put on the box with coordinates $(-s,\mu+\nu,s)$ (negative $a_s$ means removing $|a_s|$ boxes). The integer $b_s$ indicates the number of the black boxes put on the box with coordinates $(-s-1,\mu+\nu,s)$ (negative $b_s$ means removing $|b_s|$ boxes). 

The reference state is the pair of empty partitions and all integers equal to $0$.

An example of a state in $S(G_{0,1}^{\mu,\nu})$ is given in Figure \ref{slanted relaxed verma picture}. In the picture we have two vertical partitions  $(2,2,1)$ and $(2,1)$ and on the tower we have $a_0=2$, $a_1=1$, $b_1=2$.

We also show a reference state of $S(G_{0,2}^{\mu,\nu})$ in Figure \ref{Slope 2 picture}. Note that one can add boxes to the white and black colored places. In addition one can add negative boxes which amounts to taking off boxes from the staircase tower.

\begin{prop}\label{slanted relaxed Verma prop}
The space $G_{0,m}^{\mu,\nu}=V(S(G_{0,m}^{\mu,\nu}),\Psi(G_{0,m}^{\mu,\nu}))$ with action \eqref{action} is an irreducible tame  $\E$-module of level 
$q_3$. 
Moreover, for any $k\in\Z$, we have $G_{0,m}^{\mu,\nu}\simeq G_{0,m}^{\mu,\nu+k}$.
\end{prop} 
\begin{proof}
The proof is similar to the proof of Proposition \ref{Fock prop}.
\end{proof}

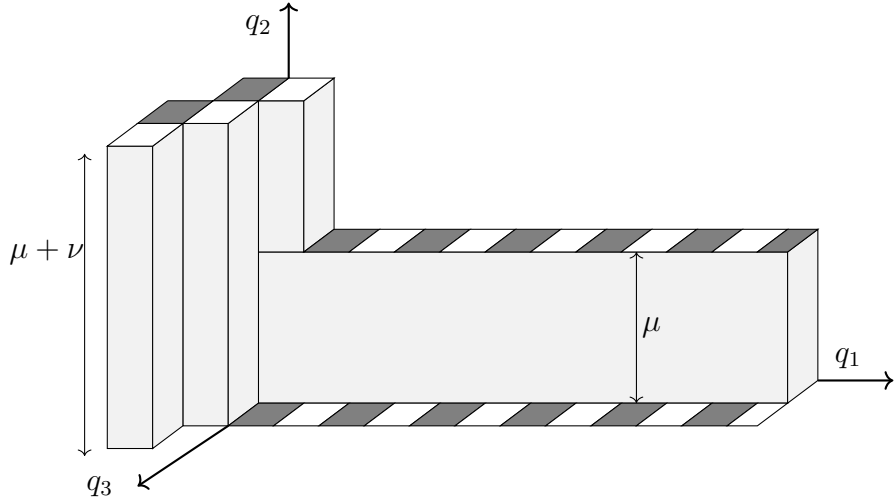
\begin{figure}[ht]
\centering
\begin{tikzpicture}
\filldraw[fill=gray!10] (0.8,-0.9)--++(0,2)--++(7,0)--++(0,-2)--(0.8,-0.9);
\filldraw[fill=gray!10] (7.8,-0.9)--++(0,2)--++(0.4,0.3)--++(0,-2)--(7.8,-0.9);
\filldraw[fill=gray!10] (0.8,1.1)--++(0,2)--++(0.6,0)--++(0,-2)--(0.8,1.1);
\filldraw[fill=gray!10] (1.4,3.1)--++(0.4,0.3)--++(0,-2)--++(-0.4,-0.3)--(1.4,3.1);
\filldraw[fill=gray!10] (-0.2,2.8)--++(0.6,0.0)--++(0,-4)--++(-0.6,0.0)--(-0.2,2.8);
\filldraw[fill=gray!10] (0.4,2.8)--++(0.4,0.3)--++(0,-4.0)--++(-0.4,-0.3)--(0.4,2.8);

\filldraw[fill=gray!10] (-1.2,2.5)--++(0.6,0.0)--++(0,-4)--++(-0.6,0.0)--(-1.2,2.5);
\filldraw[fill=gray!10] (-0.6,2.5)--++(0.4,0.3)--++(0,-4.0)--++(-0.4,-0.3)--(-0.6,2.5);


\wbox at (1.2, 3.4);
\bbox at (0.6,3.4);
\wbox at (0.2,3.1);
\bbox at (-0.4,3.1);
\wbox at (-0.8,2.8);

\bbox at (1.8,1.4);
\wbox at (2.4,1.4);
\bbox at (3.0, 1.4);
\wbox at (3.6,1.4);
\bbox at (4.2,1.4);
\wbox at (4.8, 1.4);
\bbox at (5.4,1.4);
\wbox at (6.0,1.4);
\bbox at (6.6, 1.4);
\wbox at (7.2,1.4);
{\filldraw[fill=gray,-] (7.8,1.4)--++(0.4,0)--++(-0.4,-0.3)--++(-0.4,0)--(7.8,1.4);}

\bbox at (0.8, -0.9);
\wbox at (1.4, -0.9);
\bbox at (2.0, -0.9);
\wbox at (2.6, -0.9);
\bbox at (3.2, -0.9);
\wbox at (3.8, -0.9);
\bbox at (4.4, -0.9);
\wbox at (5.0, -0.9);
\bbox at (5.6, -0.9);
\wbox at (6.2, -0.9);
\bbox at (6.8, -0.9);
{\filldraw[fill=white,-] (7.4,-0.9)--++(0.4,0)--++(-0.4,-0.3)--++(-0.4,0)--(7.4,-0.9);}

\draw[thick,->] (8.2,-0.6)--(9.2,-0.6);
\node at (8.6,-0.3)  {$q_1$};
\draw[thick,->] (1.2,3.4)--(1.2,4.4);
\node at (0.8,4.1)  {$q_2$};
\draw[thick,->] (0.4,-1.2)--(-0.8,-2);
\node at (-1.3,-2)  {$q_3$};
\draw[<->] (5.8,-0.9)--(5.8,1.1);
\node at (6,0.1)  {$\mu$};
\draw[<->] (-1.5,-1.6)--(-1.5,2.4);
\node at (-2,1.1)  {$\mu+\nu$};

\end{tikzpicture}
    \caption{The reference state in the slope 2 slanted relaxed Verma module $G_{0,2}^{\mu,\nu}$.}
\label{Slope 2 picture}
\end{figure}

\medskip

Now we compute the character of $G_{0,m}^{\mu,\nu}$. We recall our notation $t=z_0z_1$ and $\delta(z)=\sum_{i\in\Z}z^i$.

\begin{prop}\label{char prop}
We have
\be
\chi_{G_{0,m}^{\mu,\nu}}(z_0,z_1)= \frac{1}{\prod_{j=1}^\infty(1-t^j)^4}\,\delta(z_0t^m).
\ee 
\end{prop}
\begin{proof}
    We have 
    \be 
    \chi_{G_{0,m}^{\mu,\nu}}(z_0,z_1)=\bar\chi_1^2(z_0,z_1)\chi_m(z_0,z_1) =
\frac{\chi_m(z_0,z_1)}{\prod_{i=1}^\infty(1-t^i)^2(1-z_1t^{i-1})^2}
= p(t)\sum_{i\in \Z} (z_0t^m)^i,
    \ee 
    where $\chi_r(z_0,z_1)$ is the contribution of the tower 
and $p(t)$ is some series. 
    The last equality
takes place since we have $G_{0,m}^{\mu,\nu}\simeq G_{0,m}^{\mu,\nu+1}$ and 
    \be
    \chi_{G_{0,m}^{\mu,\nu+1}}(z_0,z_1)=\chi_{G_{0,m}^{\mu,\nu}}(z_0,z_1) t^mz_0=\chi_{G_{0,m}^{\mu,\nu}}(z_0,z_1).
    \ee 
    
To compute $p(t)$ we write the character as a series in $t,z_0$ and compute the coefficient of $(z_0)^0$.
    
    The state in the tower is characterized by integer numbers $b_0,\dots,b_{m-1}$, $a_0,\dots,a_m$ satisfying $b_i\geq a_i,a_{i+1}$ which tell the amount of boxes on the tower. Let $c_i=b_i-a_i$ and $d_i=b_i-a_{i+1}$, $i=0,\dots, m-1$. Then $b_i$ and $a_i$ are uniquely determined by $a_0$, $c_i$ and $d_i$. Moreover, $c_i\geq 0$ and $d_i\geq 0$.

    Note that if $a_j<0$ for some $j$, then 
    \begin{align}\label{inequality}
    \sum_{i=0}^{m-1}b_i-\sum_{i=0}^{m}a_i=\sum_{i=0}^{j-1} c_i-a_j+\sum_{i=j}^{m-1} d_i >0.
    \end{align}
    So the contribution of this state is $z_0^iz_1^j=t^jz_0^{i-j}$ with $j>i$. Since all monomials in $\bar\chi_1(z_0,z_1)$ have the form  $z_0^iz_1^j$ with $i\leq j$, such a
state does not contribute to the 
coefficient of $(z_0)^0$. Thus we can assume $a_i\geq 0$ and therefore $b_i\geq 0$ for all $i$.

For the same reason it is sufficient to consider only states with
\be
\tilde a_0=\sum_{i=0}^{m}a_i-\sum_{i=0}^{m-1}b_i=a_0-\sum_{i=0}^{m-1} d_i\geq 0.
\ee 
    
    Let the shape $R_i$ consist
of $i$ horizontal dominoes aligned to form a staircase, see Figure \ref{domino shapes}.   Adding $R_i$ to the tower aligning the white end with the box with coordinates $(0,x,0)$ changes $d_{i-1}$ to 
$d_{i-1}+1$ and $a_0$  to $a_0+1$
keeping $c_j$, and $d_k$, $k\neq i-1$ intact. 

Similarly, let shape $\bar R_i$ consist
of $i$ vertical dominoes aligned to form a staircase, see Figure \ref{domino shapes}.   Adding $\bar R_i$ to the tower aligning the white end with the box with coordinates $(-m,x,m)$ changes $c_{m-i}$ to $c_{m-i}+1$ keeping $a_0$, $d_j$, and $c_k$, $k\neq m-i$ intact.

\begin{figure}[ht]
\centering
\begin{tikzpicture}
\fwbox at (-6, 0);
\fbbox at (-6.6, 0);

\fwbox at (-3.0, 0);
\fbbox at (-3.6, 0);
\fwbox at (-3.6, -0.6);
\fbbox at (-4.2, -0.6);

\fwbox at (-0.0, 0);
\fbbox at (-0.6, 0);
\fwbox at (-0.6, -0.6);
\fbbox at (-1.2, -0.6);
\fwbox at (-1.2, -1.2);
\fbbox at (-1.8, -1.2);

\fwbox at (2, -1.8);
\fbbox at (2, -1.2);

\fwbox at (4, -1.8);
\fbbox at (4, -1.2);
\fwbox at (4.6, -1.2);
\fbbox at (4.6, -0.6);

\fwbox at (6, -1.8);
\fbbox at (6, -1.2);
\fwbox at (6.6, -1.2);
\fbbox at (6.6, -0.6);
\fwbox at (7.2, -0.6);
\fbbox at (7.2, -0.0);

\node at (-6,-2.5) {$R_1$};
\node at (-3.5,-2.5) {$R_2$};
\node at (-1,-2.5 ) {$R_3$};
\node at (2.3,-2.5) {$\bar R_1$};
\node at (4.6,-2.5) {$\bar R_2$};
\node at (7,-2.5 ) {$\bar R_3$};

\end{tikzpicture}
\caption{The shapes $R_i$ and $\bar R_i$.}
\label{domino shapes}
\end{figure}
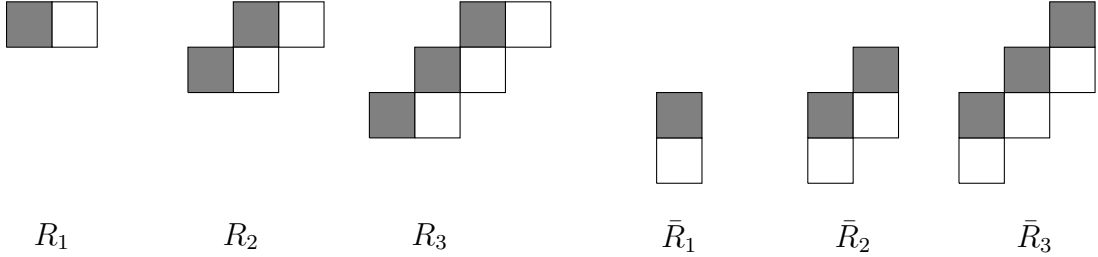
Therefore a state corresponding $a_i$ and $b_i$ can be obtained adding to the reference state $\tilde a_0$ full staircases, $d_i$ shapes $R_{i+1}$ and $c_i$ shapes $\bar R_{i}$. Thus, we have a bijection of sequences of non-negative integers $a_i,b_i$, satisfying \eqref{condition} and such that $\tilde a\geq 0$ with  sequences of non-negative integers $\tilde a_0, c_i, d_i$ without any conditions.

Thus the character of all such states is
\be
A(z_0,t)=\frac{\bar\chi_1^2(z_0,z_1)}{( 1-z_0t^m)\prod_{j=1}^{m}(1-t^{j})^2 } =\frac{1}{( 1-z_0t^m)\prod_{j=1}^{m}(1-t^{j})^2 \prod_{i=1}^\infty(1-t^i)^2(1-z_0^{-1}t^i)^2}.
\ee 
 Let $|t|<1$. Then the coefficient of $z_0$ can be computed as an integral:
 \be
p(t)=\frac{1}{2\pi\hspace{-2pt}\on{i}}\int_{|z_0|=1} A(z_0,t) \,\frac{dz_0}{z_0}=-\Res_{z_0=t^{-m}} \frac{A(z_0,t)}{z_0}=\frac{1}{\prod_{i=1}^\infty (1-t^i)^4}.
 \ee
 \end{proof}
Note that 
the character of $G_{1,m}^{\mu,\nu}$ coincides with the character of 
$\hat L_{-m}(a,b,c;\kappa)$, cf. Figure \ref{Slope 2 picture}.

The module $G_{1,m}^{\mu,\nu}$ restricted to the vertical  $U_q^v\widehat{\gl}_2\subset \E$ subalgebra is the irreducible relaxed Verma module  $\hat L_{-m}(a,b,c;\kappa)$ with
$q^a=q_3^{m+1}q^{-\mu-2\nu}$, $q^b=1$,
$\kappa=q_3$, and suitable $c$.  
Moreover, $G_{1,m}^{\mu,\nu}$ is the evaluation module obtained from $\hat L_{-m}(a,b,c;\kappa)$.

Similarly, the module $G_{0,m}^{\mu,\nu}$ is the evaluation module obtained from slope $m+1$
relaxed module $\hat L_{m+1}(a,b,c;\kappa)$.

\bigskip

{\bf Acknowledgments.\ } MJ is partially supported by 
JSPS KAKENHI Grant Numbers \newline JP19K03509, JP20K03568.
EM is partially supported by Simons Foundation grant number \#709444.

EM thanks Rikkyo University, where a part of this work was done, 
for hospitality.

We thank B. Feigin for interesting discussions.

\end{document}